\documentclass[11pt,a4paper,final]{amsart}
\usepackage[latin1]{inputenc}
\usepackage{amsmath}
\usepackage{amsfonts}
\usepackage{amssymb}
\usepackage{amsthm}
\usepackage{mathrsfs}
\usepackage{graphicx}
\usepackage{a4wide}
\usepackage{bbm}
\usepackage{todo}
\usepackage{latexsym}
\usepackage{cite}
\usepackage{mathrsfs}
\usepackage{bbm}
\usepackage[dvipsnames]{xcolor}
\usepackage{subfigure}
\usepackage{lmodern}
\usepackage{float}
\usepackage{epsfig}
\usepackage{epstopdf}
\usepackage{mathtools}
\usepackage{listings}
\usepackage{pgf}
\usepackage[T1]{fontenc}
\usepackage[latin1]{inputenc}
\usepackage[english]{babel}
\usepackage{pifont,tabularx}
\usepackage{stmaryrd}
\usepackage{dsfont}
\usepackage{theoremref}
\usepackage{nicefrac}
\usepackage{enumerate}
\usepackage{natbib}
\usepackage{subfigure}

\usepackage{tikz}
\usepackage{pgfplots}
\usepackage{pgfplotstable}
\pgfplotsset{compat=1.7}
\usepackage{caption}
\usepackage{hyperref}
\hypersetup{
	colorlinks,
	citecolor=blue,
	linkcolor=blue,
	urlcolor=teal}

\numberwithin{equation}{section}

\theoremstyle{plain} 
\newtheorem{thm}{Theorem}[section]
\newtheorem{prop}[thm]{Proposition}
\newtheorem{cor}[thm]{Corollary}
\newtheorem{lem}[thm]{Lemma}

\newtheorem{obs}[thm]{Remark}

\DeclareMathOperator*{\argmin}{arg\,min}

\theoremstyle{remark}

\usepackage{graphicx} % Required for inserting images

\title{Mixing time and isoperimetry of long-range percolation in the diffusive regime}

\author{Dieter Mitsche}
\address{IMC, Pontificia Universidad Catolica de Chile}
\curraddr{Avenida Vicu\~na Mackenna 4860, Santiago, Chile}
\email{dieter.mitsche@uc.cl}
\thanks{}

\author{Carlo Scali}
\address{School of Computation, Information and Technology, Technische Universit{\"a}t M{\"u}nchen}
\curraddr{Boltzmannstra{\ss}e 3, 85748, Garching bei M{\"u}nchen}
\email{carlo.scali@tum.de}
\thanks{}

\date{\today}

\begin{document}
\begin{abstract}
We consider long-range percolation (LRP) on $\mathbb{Z}^d, d \ge 1$ with connection probabilities $\mathbf{p}(x, y) \approx \frac{\beta}{\|x-y\|^s}$ for $\beta>0$ and $s > \min(2d, d+2)$. Our main result is that the mixing time of the random walk on the giant component of supercritical LRP inside a box of side length $n$ is of order $n^2$. To establish the upper bound, we show that a $d$-dimensional isoperimetric inequality holds almost surely simultaneously for all connected sets inside the giant component that are large enough.
\smallskip


\end{abstract}
\maketitle

\section{Introduction}

The study of mixing times on random graphs provides a fundamental connection between the structural properties of random networks and the dynamics of stochastic processes defined on them. In particular, mixing time quantifies how rapidly a random walk approaches its stationary distribution, thereby measuring the efficiency with which information or randomness spreads throughout the graph. Understanding this behaviour is important for characterising the interplay between randomness, graph structure, and diffusion, with applications in probability, combinatorics, network science, and theoretical computer science. The mixing time was studied on the classical percolation model, see~\cite{BenjaminiMossel}.

Of particular interest is the case of percolation, and long-range percolation (LRP) is a particularly rich model in this context. By allowing connections between distant vertices with probabilities that decay polynomially with their Euclidean distance, it captures a remarkable spectrum of geometric behaviours and interpolates between regimes that retain strong Euclidean characteristics and regimes exhibiting pronounced small-world phenomena. Its geometric structure has been the subject of extensive investigation, and substantial progress has been made in understanding the existence and uniqueness of the infinite cluster, graph distances and diameters across the different regimes \cite{AizenmanNewman, BergerLRP, BiskupDistance1, BiskupDiameter, BiskupOscillation, BaumlerDistance1, BaumlerShapeTheo}. The behaviour of random walks on these networks is understood to a lesser extent: scaling-limit, mixing-time and heat-kernel results are available in several parameter regimes \cite{BergerBenjaminiYadin, CrawfordSly1, CrawfordSly2, BiskupKumagaiQuenched, KumagaiLocalLimit, CroydonCanKumagai}, but a complete quantitative picture is still missing. We also mention the recent breakthrough \cite{SmallWorldLRP}, in the regime in which LRP exhibits small world phenomena and \cite{BiskupKumagaiQuenched, KumagaiLocalLimit} in the same range of parameters that we study here.

\subsection{Model and main results} In this paper, we study the random graph induced by the classical long-range percolation (LRP) model on finite boxes of $\mathbb{Z}^d$ in the supercritical regime. We will denote by $\omega = \{\omega(x, y)\}_{x, y \in \mathbb{Z}^d}$ the states of the edges (i.e.\ the resulting random graph). Note that $\omega \in \Omega \in \{0, 1\}^{\mathbb{Z}^d \times \mathbb{Z}^d}$. Let $s \in (d, \infty)$, $x, y \in \mathbb{Z}^d$, $\beta>0$, and consider the probabilities
\begin{equation}\label{eqn:ConnectionProbs}
    \mathbf{p}(x, y) \coloneqq \mathbf{P}\left( \omega(x, y) = 1 \right) \sim \frac{\beta}{\|x-y\|^{s}},
\end{equation}
 where $f(x) \sim g(x)$ if $\lim_{x \to \infty} f(x)/g(x) = 1$. The Bernoulli random variables associated to the edges are independent. It is known \cite{AizenmanNewman, BergerLRP} that for every $s \in (d, \infty)$ there exists $\beta_c=\beta_c(s,d)>0$ such that the random graph percolates for all $\beta>\beta_c$. We assume throughout $\beta > \beta_c$ so that
\begin{equation*}
    \mathbf{P}\left(0 \leftrightarrow \infty\right) = \theta(\beta, s)>0.
\end{equation*}
Of particular significance is the measure $\mathbf{P}\left( \cdot \mid 0 \leftrightarrow \infty \right)$, where the percolation configuration is chosen conditioned on the event (of strictly positive probability) that the origin lies in the unique infinite connected component denoted $\mathcal{C}_\infty$. We observe that in case $d=1$ and $s>2$, one needs to assume that all nearest neighbour edges (assuming translation-invariant probabilities as above) are present to have an infinite component. We will assume this when dealing with the $1$-dimensional case in Section~\ref{sect:D1}.

For the degree of a vertex $x$, let us set $\mathrm{deg}^\omega(x) \coloneqq \sum_{z \in \mathbb{Z}^d}\omega(x, z)$ (for a general graph $G$, $\mathrm{deg}(v)$ is the degree of a vertex $v$). We define then a Markov chain $\{X_{t}\}_{t \ge 0}$ in continuous time, often called the \emph{constant speed random walk} (CSRW): on a fixed realisation of the environment $\omega$, we study the random walk, i.e.\ the continuous time Markov chain with \textbf{quenched} generator
\begin{equation*}
    (\mathcal{L}^\omega f)(x) \coloneqq \frac{1}{\mathrm{deg}^\omega(x)} \sum_{y \sim x} \omega(x, y)(f(y) - f(x)).
\end{equation*}
The normalisation of the rate by $\mathrm{deg}^\omega(x)^{-1}$ makes the random walk jump at i.i.d.\ exponential times. More precisely, there exists a sequence $\{T_n\}_{n \ge 0}$ such that $T_{n+1}- T_n, n \ge 0$ are i.i.d.\ exponential random variables of parameter $1$ and $X_{T_n}$ is the position of the random walk after the $n$-th jump. If we evaluate the resulting random walk $\{X_{t}\}_{t \ge 0}$ at the discretised times, and write $\widetilde{X}_n \equiv X_{T_n}$, the resulting discrete time Markov chain, for all $x, y \in \mathbb{Z}^d$ and any starting point $z \in \mathbb{Z}^d$ in the same connected component as $x$ of the process, has jump probabilities
\begin{equation*}
     P^\omega_z\left( \widetilde{X}_{n+1} = y \mid \widetilde{X}_n = x \right) = \frac{\omega(x, y)}{\sum_{r \in \mathbb{Z}^d}\omega(x, r)} = \frac{\omega(x, y)}{\mathrm{deg}^\omega(x)}. 
\end{equation*}
We will refer to the random walk $\{\widetilde{X}_{n}\}_{n \ge 0}$ as the \emph{embedded walk}. The point $z \in \mathbb{Z}^d$ denotes in the notation $P_z^\omega$ the almost sure starting point of the process, which typically for us will be the origin, and we call it $P_0^\omega$. For $n \in \mathbb{N}$, let $p_n^\omega(x, y)$ denote the probability that the random walk with jump probabilities $p^\omega(\cdot, \cdot)$ is at $y$ starting from $x$ in $n$ steps. Then, we introduce the \emph{heat-kernel} of $X$:
\begin{equation*}
    p_t^\omega(x, y) = \sum_{n \ge 0} \frac{t^n e^{-t}}{n!} p_n^\omega(x, y),
\end{equation*}
and we define the \textbf{annealed law} of the random walk
\begin{equation*}
    \mathbb{P}_{0}\left( X_t \in \cdot \right) \coloneqq \mathbf{E}\left[P_0^\omega\left(X_t\in \cdot\right)\right].
\end{equation*}

Let $G = (V, E)$ be an undirected finite graph with unit weights. For each vertex $v \in V$, we denote by
\begin{equation*}
    \pi(v) \coloneqq \frac{\mathrm{deg}(v)}{2|E|},
\end{equation*}
the stationary measure for the constant speed random walk in continuous time. Then, we define the \textbf{mixing time} for $\varepsilon>0$:
\begin{equation*}
    T_{\mathrm{mix}}(\varepsilon) \coloneqq \inf\left\{t \ge 0 \colon \sup_{x \in V} \frac{1}{2}\sum_{y \in V} |p^\omega_t(x, y ) - \pi(y) | \le \varepsilon \right\}.
\end{equation*}
We will also set $T_{\mathrm{mix}} = T_{\mathrm{mix}}(1/4)$. We will denote by $G_n = (V_n, E_n)$ the graph obtained by the LRP on $\Lambda_n \coloneqq [-n/2, n/2]^d \cap \mathbb{Z}^d$. Moreover, we let $\mathcal{C}^{(n)}_1$ be the largest connected component inside $G_n$ (with an arbitrary deterministic rule to break ties). The main theorem of the present article is the following theorem:

\begin{thm}\label{theo:MixingDiffusive}
    Assume $d\ge 1$ and $s>\min\{d+2, 2d\}$ in long-range percolation on $\Lambda_n$. For all $\varepsilon>0$, there exist two constants $c_\varepsilon, C_\varepsilon>0$, such that with $\mathbf{P}$-probability at least $1-\varepsilon$ we have that
    \begin{equation*}
        c_{\varepsilon} n^2 \le T_{\mathrm{mix}} \le C_{\varepsilon} n^2,
    \end{equation*}
    where $T_{\mathrm{mix}}$ denotes the mixing time of $\mathcal{C}^{(n)}_1$.
\end{thm}

The upper bound in  Theorem \ref{theo:MixingDiffusive} will be established through the well-known connection between isoperimetric profiles and mixing times (see Section \ref{sect:Cond}). We summarise our results on the isoperimetric properties of connected sets in the LRP in the following theorem:

\begin{thm}\label{thm:LBIsoGen}
     Let $d\ge 2$ and $s>\min\{d+2, 2d\}$. There exists a constant $c_{\mathrm{low}}>0$ such that for all $\eta>0$ there exists $c>0$ such that, for all $n$ large enough, with $\mathbf{P}$-probability at least $1 - \exp\{-c \log^2(n)\}$ for all $A \subseteq \mathcal{C}^{(n)}_1 \colon n^{\eta} \le |E(A, \Lambda_n)| \le |E(\mathcal{C}^{(n)}_1, \mathcal{C}^{(n)}_1)|/2$ we have that
\begin{equation*}
    |E(A, A^c)| \ge c_{\mathrm{low}} |E(A, \Lambda_n)|^{\frac{d-1}{d}}.
\end{equation*}
\end{thm}

\begin{obs}
    By optimising our arguments, one can show that Theorem \ref{thm:LBIsoGen} is valid for $A \subseteq \mathcal{C}^{(n)}_1$ satisfying $ \log^\Delta(n) \le |E(A, \Lambda_n)| \le |E(\mathcal{C}^{(n)}_1, \mathcal{C}^{(n)}_1)|/2$, for some $\Delta>0$ large enough. Since we cannot obtain the exact value of $\Delta$, and the polynomial bound is enough for our purpose, we did not pursue this. Determining the sharp $\Delta>0$ (and the estimates required to identify it) might be an interesting direction for future research.
\end{obs}

\subsection{Overview of the proof and related work}
Whilst our results perhaps are not surprising by the similarity of the regime $s > d+2$ with Bernoulli percolation, the proof techniques are considerably different. In particular, most of the difficulty in the proofs arises from the fact that we do \emph{not} insist on having nearest-neighbor edges being present: for $d=1$ we have to assume this in order to ensure supercriticality, and the proof is much simpler, as we show at the end of this paper. Assuming the presence of nearest neighbour edges in higher dimensions would also significantly simplify the argument also in higher dimensions, in particular we would not require the analysis of Section \ref{sect:Prep}.

The lower bound is comparably simpler and based on a reversibility trick of Kesten, which ensures that the random walk is not faster than diffusive. On the other hand, the upper bound requires several steps: we first use a truncated supercritical model as a building block. The analyisis of the truncated model is similar to Bernoulli percolation in \cite{PeteIso} or the recent works on random geometric graphs \cite{MitscheMartinezKiwi, HaalandSarkovic}. The crucial input from \cite{BergerLRP, BaumlerShapeTheo} is the possibility to truncate the LRP, while still having supercritical percolation. Then, inspired by Crawford-Sly \cite{CrawfordSly1}, the crucial part of our proof is an allocation theorem. This result allows us to allocate small clusters that are subcritical in the truncated model, but are attached to the giant component in the full model. This step is required to ensure that no additional polylogarithmic factors appear when computing the edge-volume of connected sets. These preparations are then used when showing the isoperimetric inequalities.\\

\noindent \textbf{Notation.}
For $n \in \mathbb{N}$ and $y \in \mathbb{Z}^d$, we define
\begin{equation*}
    \Lambda_n(y) \coloneqq \left\{x \in \mathbb{Z}^d \colon \|x - y\|_\infty \le \tfrac{n}{2}\right\} = y + \Lambda_n.
\end{equation*}
Note that $\Lambda_n(0)=\Lambda_n$.
%Observe that $\Lambda_n = \Lambda_n(0) = [-n, n]^d \cap \mathbb{Z}^d$. 
For a graph $G = (V, E)$, let us write $d_{G}(x, y)$ for the graph (or chemical) distance between a pair of vertices $x, y \in V$. We use the convention that $d_{G}(x, y) = +\infty$ if $x, y$ are not connected by a path of open edges. We will denote by $\mathcal{C}_\infty$ the almost surely unique infinite connected component on $\mathbb{Z}^d$. 
For a set $A \subseteq \mathbb{Z}^d$ we denote
\begin{equation}\label{eqn:HittingTime}
    \tau_A \coloneqq \inf\{s \ge 0 \colon X_s \in A\}.
\end{equation}
Furthermore, for $n \in \mathbb{N}$, considering the random walk $X_t$ on $\mathcal{C}_\infty$, we set 
    \begin{equation}\label{eqn:ExitCTRW}
        \tau_{n} \coloneqq \inf\left\{s \ge 0 \colon X_s \in \Lambda_n^c\right\},
    \end{equation}
and similarly
\begin{equation}\label{eqn:ExitDTRW}
    \widetilde{\tau}_{n} \coloneqq \inf\left\{k \ge 0 \colon \widetilde{X}_k \in \Lambda_n^c\right\}.
\end{equation}
We will denote by 
\begin{equation*}
    \mathcal{C}^{(n)}_j \coloneqq \left\{\text{the }j\text{-th largest connected component inside }\Lambda_n\right\}.
\end{equation*}
Let $\mathcal{C}_\infty|_{\Lambda_n}$ be the largest component of $\mathcal{C}_\infty$ connected inside the box $\Lambda_n$, with some deterministic rule to break ties. We remark that, in general, $\mathcal{C}_\infty|_{\Lambda_n} \neq \mathcal{C}^{(n)}_1$, but we will show  that in fact $\mathcal{C}_\infty|_{\Lambda_n} \supseteq \mathcal{C}^{(n)}_1$ with suitably high probability. Furthermore, for all $k \in \mathbb{N}$, we let $\mathcal{C}^{(k)}_j(x)$ be the $j$-th largest component of LRP inside the box $\Lambda_k(x)$. Define in the analogous way $\mathcal{C}_\infty|_{\Lambda_k(x)}$.

For technical reasons, let us denote by $\mu$ the probability measure generating the percolation configuration and enhance the notation and write $\mathbf{P}^\mu$. The reader should always assume that $\mathbf{P} = \mathbf{P}^\mu$; the second version will be exclusively used when we wish to highlight the usage of $\mu$ among the other measures that we proceed to introduce now. Let us write $\mu_0(\cdot) \coloneqq \mu\left( \cdot \mid 0 \in\mathcal{C}_\infty \right)$. Let $\nu$ be the measure on environments such that for any measurable $A \subseteq \{0, 1\}^{\mathbb{Z}^d}$, 
\begin{equation}\label{eqn:DefStationaryMeasure}
    \mathbf{P}^{\nu}(A) = \frac{\mathbf{E}[\deg(0) \mathds{1}\{A\}]}{\mathbf{E}[\deg(0)]}.
\end{equation}
It is a well-known fact that the measure $\mathbf{P}^{\nu}(\cdot)$ is stationary and ergodic for the random walk \cite[pg.~449]{CrawfordSly2}. We will denote by $\mathbb{P}^\nu$ the annealed law of the walk under environments chosen according to $\nu$. We also introduce the measure $\nu_0(\cdot) =  \nu\left( \cdot \mid 0 \in\mathcal{C}_\infty \right)$.

We will say that a sequence of events $\{B_n\}_{n \ge 0}$ holds with high probability (w.h.p.), when $\mathbf{P}(B_n) \ge 1 - \varepsilon$ for any $\varepsilon>0$ and any $n$ large enough. 
We use standard asymptotic Landau notation throughout the paper. Finally, constants $c, c'$  change from line to line, whereas constants with subindices do not.\\ 

\noindent \textbf{Organisation.} In Section \ref{sect:Notation} we collect useful definitions and previous results on the model. Section \ref{sect:LB} is then dedicated to  the lower bound for the mixing time. Next, in Section \ref{sect:Trunc} we study the truncated percolation model, proving Proposition \ref{prop:IsoTruncated}. In Section \ref{sect:Prep} we show the main auxiliary result, Theorem \ref{theo:Partition}. In Section \ref{sect:UB} we conclude the proof of Theorem \ref{thm:LBIsoGen} and deduce Theorem \ref{theo:MixingDiffusive} for $d \ge 2$. Finally, in Section \ref{sect:D1} we show Theorem \ref{theo:MixingDiffusive} for $d=1$. We assume $d \ge 2$ and $s> d+2$ throughout the paper with the only exception of Section \ref{sect:D1}.

\section{Preliminary results}\label{sect:Notation}

\subsection{Tilings} Here, we set the notation for certain coverings of boxes by sub-boxes (alternatively called tiles). For $n>0$ fixed and $N_1 < n$, we let for $x \in \mathbb{Z}^d$
\begin{equation*}
    \{\Lambda_{N_1}\}_{\Lambda_{N_1} \in \mathcal{T}_{N_1}}
\end{equation*}
be a fixed minimal cover of $\Lambda_n$ by vertex-disjoint boxes of side length $N_1$. If $N_1$ is not a divisor of $n$ the minimal cover might not be unique, we choose one by some deterministic rule. We remark that the boxes $\Lambda_{N_1} \in \mathcal{T}_{N_1}$ are of the form $\Lambda_{N_1}(x)$ for some $x \in \mathbb{Z}^d$. We say that two boxes $\Lambda_{N_1}(xN_1)$ and $\Lambda_{N_1}(yN_1)$ are adjacent if $\|x-y\|_1\le 1$, and furthermore, we call them $*$-adjacent if $\|x-y\|_\infty\le 1$.

%\cs{Maybe here add overlapping tiling definition.}\dmc{At this point we don't need them yet, no?}
\subsection{The conductance method} \label{sect:Cond} For any $A \subseteq V$ we will write $\pi(A) = \sum_{v \in A} \pi(v)$. Moreover, for a pair of sets $A, B \subseteq V$ we denote $E(A, B)$ the set of edges with one endpoint in $A$ and one in $B$.

We quantify the edge boundary of a set $A$ through the quantity
\begin{equation*}
    Q(A, A^c) \coloneqq \frac{|E(A, A^c)|}{2 |E|}.
\end{equation*}
From this we define the conductance function, defined for $t \in (0, 1)$ as
\begin{equation*}
    \varphi(t) \coloneqq \inf_{\substack{A \subseteq V \\ 0 \le \pi(A) \le t}} \frac{Q(A, A^c)}{\pi(A) \pi(A^c)}.
\end{equation*}
This quantity is famously connected to the mixing time of the random walk on $G$ by the identity (this first appeared in \cite[Theorem 2.1]{ConductanceMethod} and was later corrected by Montenegro~\cite[Theorem 3.1]{montenegro2002faster}):
\begin{equation}\label{eqn:CondMeth}
    T_{\mathrm{mix}} \le \int_{\pi_0}^{1/2} \frac{1}{\varphi^2(t)t} dt + \frac{1}{\varphi(1/2)},
\end{equation}
with $\pi_0 = \min_{v \in V} \pi(v)$. Commonly, the second term on the right-hand-side is not relevant in proving the upper bound, and this will be indeed the case for us.

\subsection{Auxiliary lemmas}

The first result ensures that with high probability the largest component occupies a positive density of a large box. 

\begin{thm}\cite[Theorem~1.8, (28)]{BaumlerShapeTheo}\label{theo:Baumler}
For every $\varepsilon > 0$ there exists $n_0 > 0$ such that for all $n \ge n_0$
\begin{equation*}
    \mathbf{P}\left(|\mathcal{C}^{(n)}_1| \le (\theta(\beta, s) - \varepsilon) |\Lambda_n| \right) \le \varepsilon.
\end{equation*}
Moreover, there exists $\eta>0$ such that for $n$ large enough 
\begin{equation*}
    \mathbf{P}\left(\mathcal{C}_\infty \cap \Lambda_n  = \emptyset \right) \le e^{-n^{\eta}}.
\end{equation*}
\end{thm}

\begin{lem}\label{lemma:SecondLargest}
    There exists $\zeta>0$ such that, $\mathbf{P}$-almost surely for all $n$ large enough 
    \begin{equation*}
        |\mathcal{C}_2^{(n)}| \le \log(n)^\zeta.
    \end{equation*}
    As a consequence, for any $\varepsilon>0$, we obtain that for all $m$ large enough and all $x \in \mathbb{Z}^d$,
    \begin{equation*}
        \mathbf{P}\left(\mathcal{C}^{(m)}_1(x) \not \subseteq \mathcal{C}_\infty|_{\Lambda_m(x)} \right) \le \varepsilon.
    \end{equation*}
\end{lem}
\begin{proof}
    The first result will be proved separately for the two regimes $s \in (d, d+2)$ and $s \ge d+2$, for $d \ge 2$ (see Theorem \ref{theo:Partition}). 
    
    The second result follows from the first and Theorem~\ref{theo:Baumler}. Indeed, if $\mathcal{C}^{(m)}_1$ is larger than $\theta(\beta, s)m^d/2 \ge \theta(\beta, s)n^{d/2}/2$, then $\mathcal{C}^{(m)}_1 = \mathcal{C}_\infty|_{\Lambda_m(x)}$ by the first part of the lemma.
\end{proof}
We observe that Lemma~\ref{lemma:SecondLargest} implies that for all $n$ large enough
\begin{equation*}
    \mathbf{P}\left(0 \in \mathcal{C}^{(n)}_1 \right) \ge \frac{\theta(\beta, s)}{2}.
\end{equation*}
We quote the following result from \cite{CrawfordSly1}.
\begin{lem}\cite[Lemma~2.10]{CrawfordSly1}\label{lem:MaxDeg}
    For every $\delta>0$, there exists $C=C(\delta)>0$ such that for all $b>0$ and $n \ge 1$,
    \begin{equation*}
        \mathbf{P}\left( \max_{x \in \Lambda_n} \mathrm{deg}^\omega(x) > b \log(n)\right)\le Cn^d e^{-\delta b \log(n)}.
    \end{equation*}
    This immediately implies that there exists $U > 0$ such that $\mathbf{P}$-almost surely for all $n \ge 0$ large enough, $\max_{x \in \Lambda_n} \mathrm{deg}^\omega(x) \le U \log(n)$.
\end{lem}
The next lemma guarantees that with exponentially large probability large sets do not have excessively high degree. 
\begin{lem}\label{lemma:VolLargeSets}
    Let $s>d$. For all $\mu>0$ large enough, there exists a constant $c(\mu)>0$ such that for any set of vertices $A \subseteq \Lambda_n$ we have
    \begin{equation*}
        \mathbf{P}\left(\sum_{x \in A} \mathrm{deg}^\omega(x) \ge \mu |A|\right) \le e^{- c(\mu) |A|}.
    \end{equation*}
    Furthermore, we have that $c(\mu) \to \infty$ as $\mu \to \infty$. 
\end{lem}
\begin{proof}
Let $\tilde{E} \coloneqq A \times \Lambda_n$ be the set of potential edges with at least one endpoint inside $A$. We observe that
\begin{equation*}
    \sum_{x \in A} \mathrm{deg}^\omega(x) \le  2 \sum_{e \in \tilde{E}} \mathds{1}\{\omega_e = 1\}.
\end{equation*}
We observe that the degree of a vertex has exponential moments (in the infinite space), hence in particular we may denote the expected degree by $0 < \bar{d} < \infty$ (in the infinite space). Clearly, for all $n$ large enough we have that
\begin{equation*}
    \mathbf{E}\left[\sum_{e \in \tilde{E}} \mathds{1}\{\omega_e = 1\}\right] \le \bar{d}  |A|.
\end{equation*}
Moreover, from the basic fact that for a Bernoulli random variable the variance is less than the mean, 
\begin{equation*}
    \mathbf{E}\left[\sum_{e \in \tilde{E}}\left( \mathds{1}\{\omega(e) = 1\} - \mathbf{p}(e) \right)^2\right] \le \bar{d} |A|.
\end{equation*}
Then, by Bernstein's inequality (see e.g.\ \cite[Theorem~2.9.1]{HighDim}),
\begin{align*}
    \mathbf{P}\left(\sum_{x \in A} \mathrm{deg}^\omega(x) \ge \mu |A|\right) &\le \mathbf{P}\left( 2\sum_{e \in \tilde{E}} \left(\mathds{1}\{\omega_e = 1\} - \mathbf{p}(e)\right) \ge \left(\mu - \bar{d}\right) |A|\right)\\
    & \le \exp\left( - c \frac{\left(\mu - \bar{d}\right)^2 |A|^2}{\bar{d} |A| +\left(\mu - \bar{d}\right) |A|/3}\right),
\end{align*}
yielding the result for all $\mu> \bar{d}$ large enough. The second part of the lemma follows directly from this bound.
\end{proof}

\section{The lower bound on the mixing time}\label{sect:LB}

To prove the lower bound on the mixing time we will mostly work with the embedded walk $\{\widetilde{X}_k\}_{k \ge 0}$. We will show that it takes it an amount of time of order $n^2$ to exit $\Lambda_{\delta n}$ (for $\delta>0$ small), and hence, the walk cannot be mixed by that time by~\cite[Theorem~1.1]{PeresSousi}. Translating the result to the CTRW $\{X_t\}_{t \ge 0}$ is then routine using \cite[Theorem~2]{Oliveira} which states a result akin to \cite[Theorem~1.1]{PeresSousi} for continuous time random walks. Observe that this is only marginally stronger to proving \emph{tightness} for the random walk paths with diffusive scaling.

We highlight to the reader's benefit that, in this section, we will work with the random walk $\{\widetilde{X}_k\}_{k \ge 0}$ both on the infinite component $\mathcal{C}_\infty$ and on $\mathcal{C}_1^{(n)}$. With slight abuse of notation, we will maintain the same letter for both random walks. The distinction will be made case by case by the conditioning either on the event $\{0 \in \mathcal{C}_\infty\}$ or $\{0 \in \mathcal{C}_1^{(n)}\}$, and a coupling between these two random walks will be established along the proof.

The following key result follows from a ``reversibility trick'' of Kesten \cite{KestenGWT}, also used in \cite{BergerTokushige}. Recall the quantity $\widetilde{\tau}_{n}$ defined in \eqref{eqn:ExitDTRW}.

\begin{prop}\label{prop:AnnealedExit}
    Let $s> d+2$. For all $\varepsilon>0$ there exists $c_\varepsilon > 0$ such that under the annealed law we have 
    \begin{equation*}
        \mathbb{P}_0\left( \widetilde{\tau}_{n} \le c_{\varepsilon} n^2 \right) \le \varepsilon.
    \end{equation*}
    The same bound also holds under the conditioned measure $\mathbb{P}_0(\cdot \mid 0 \in \mathcal{C}_\infty)$
\end{prop}
\begin{proof}
Recall that the measure $\mathbf{P}^{\nu}(\cdot)$ defined in \eqref{eqn:DefStationaryMeasure} is stationary and ergodic for the random walk. We will denote by $\mathbb{P}^\nu$ the annealed law of the walk under environments chosen according to $\nu$. We begin by showing that the required bound holds under $\mathbb{P}^\nu$: 
    \begin{equation*}
        \mathbb{P}_0^\nu\left( \max_{k \le n} \|\widetilde{X}_n\|^2 > M n \right) \le \varepsilon.
    \end{equation*}
    To this end, let us define by 
    \begin{equation*}
        d^\omega(x) \coloneqq \sum_{y \in \mathbb{Z}^d} (y-x) p^\omega(x, x + y),
    \end{equation*}
    the local drifts. Then we notice that
    \begin{equation*}
        M_{\ell} \coloneqq \sum_{k = 1}^\ell (\widetilde{X}_{k} - \widetilde{X}_{k-1}) - \sum_{k = 1}^\ell d^\omega(\widetilde{X}_{k-1}),
    \end{equation*}
    is a quenched martingale. Note that the variance of the annealed increments of $M_{\ell}$ is finite as
    \begin{equation*}
        \sum_{y \in \mathbb{Z}^d} (y-x)^2 p^\omega(x, x + y) = \beta \sum_{k = 1}^\infty k^2 k^{d-1} k^{-s} < \infty.
    \end{equation*}
    Let $Y_k \coloneqq \widetilde{X}_{n} - \widetilde{X}_{n-k}$, $k = 0, \dots, n$. Then we notice that 
    \begin{equation*}
        X_n = \frac{1}{2}(\widetilde{X}_n + Y_n).
    \end{equation*} 
    We can also define
    \begin{equation*}
        M'_{\ell} \coloneqq \sum_{k = 1}^\ell (Y_{k} - Y_{k-1}) - \sum_{k = 1}^\ell d^\omega(\widetilde{X}_{n-k}),
    \end{equation*}
    which is also a martingale with the same distribution as $M_\ell$ (but adapted to the time-reversed filtration). Furthermore, we obtain
    \begin{align*}
        \|\widetilde{X}_n\|^2 &= \frac{1}{4} \|M_n + M_n' - d^\omega(\widetilde{X}_{0}) + d^\omega(\widetilde{X}_{n})\|^2 \\
        & \le \|M_n \|^2 + \|M_n'\|^2 + \|d^\omega(\widetilde{X}_{0})\|^2 + \|d^\omega(\widetilde{X}_{n})\|^2.
    \end{align*}
    The last two terms are not problematic for us as their expectation is constant. For the first two, since both are $L^2$ martingales, we have for some $\sigma^2<\infty$
    \begin{equation*}
        \mathbb{E}^\nu_0\left[\|M_n \|^2\right] \le \sigma^2 n \mbox{ and } \mathbb{E}^\nu_0\left[\|M_n' \|^2\right] \le \sigma^2 n.
    \end{equation*}
    Then, Doob's maximal inequality yields
    \begin{equation*}
        \mathbb{P}^\nu_0\left(\max_{k \le n} \|\widetilde{X}_k\|^2 > M n\right)\le \frac{4\mathbb{E}^\nu_0\left[\|M_n \|^2\right]}{Mn} \le \frac{4\sigma^2}{M}.
    \end{equation*}
     To conclude the proof under $\mathbb{P}^\nu$, we observe that
    \begin{equation*}
        \mathbb{P}_0^{\nu}\left( \widetilde{\tau}_{n} \le c_{\varepsilon} n^2 \right) = \mathbb{P}_0^\nu\left( \max_{k \le c_\varepsilon n^2} \|\widetilde{X}_k\|^2 > n^2 \right) \le \varepsilon,
    \end{equation*}
    where we can choose $c_\varepsilon = 1/M$ and then apply the estimates above.
    Furthermore, recalling that $\mu$ is the original measure of the environment, we notice that $\{\max_{k \le n} \|X_k\|^2 > M n\} \cap \{\mathrm{deg}^\omega(0) = 0\} = \emptyset$. Hence, we obtain
    \begin{equation*}
        \mathbb{P}^\nu_0\left(\max_{k \le n} \|\widetilde{X}_k\|^2 > M n\right) \ge \frac{1}{\mathbf{E}^\mu[\mathrm{deg}^\omega(0)]} \mathbb{P}^\mu_0\left(\max_{k \le n} \|\widetilde{X}_k\|^2 > M n\right) = c \mathbb{P}^\mu_0\left(\max_{k \le n} \|\widetilde{X}_k\|^2 > M n\right),
    \end{equation*}
    yielding the result under $\mathbb{P}^\mu$.
    
    To deduce the result under the conditioned measure, recall that $\mu_0, \nu_0$ be the two environment measures conditional on $0$ being in the infinite cluster. To transfer the bound to $\mu_0$ we can apply \cite[Lemma 9.4]{CrawfordSly2} and conclude that there exist $C_1, C_2 > 0$ such that for any event $A$,
    \begin{equation*}
        \mathbb{P}^{\mu_0}_0(A) \le C_1 \mathbb{P}^{\nu_0}_0(A) \le C_2 C_1\mathbb{P}^{\nu}_0(A).
    \end{equation*}
\end{proof}

\begin{cor}\label{cor:QuenchedExit}
    For all $\varepsilon>0$ there exists $c'_\varepsilon > 0$ such that under we have 
    \begin{equation*}
        P^\omega_0\left( \widetilde{\tau}_{n} \le c'_{\varepsilon} n^2 \right) \le \varepsilon,
    \end{equation*}
    with $\mathbf{P}$-probability larger than $1-\varepsilon$, the same result holds for $\mathbf{P}( \cdot \mid 0 \in \mathcal{C}_\infty)$-probability larger than $1-\varepsilon$.
\end{cor}

\begin{proof}
    From Proposition~\ref{prop:AnnealedExit} we obtain that
    \begin{equation*}
        \mathbf{E}\left[P_0^\omega\left( \widetilde{\tau}_{n} \le c_{\varepsilon} n^2 \right)\right] \le \varepsilon.
    \end{equation*}
    Set $c'_{\varepsilon} = c_{\varepsilon^2}$. Then, by Markov's inequality we have
    \begin{equation*}
        \mathbf{P}\left(P_0^\omega\left( \widetilde{\tau}_{n} \le c_{\varepsilon^2} n^2 \right) > \varepsilon \right) \le \frac{\mathbf{E}\left[P_0^\omega\left( \widetilde{\tau}_{n} \le c_{\varepsilon^2} n^2 \right)\right]}{\varepsilon} \le \varepsilon.
    \end{equation*}
    The same proof works for $\mathbf{P}( \cdot \mid 0 \in \mathcal{C}_\infty)$.
\end{proof}

We are now in position to give the proof of the first part of our main result.

\begin{proof}[Proof of Theorem~\ref{theo:MixingDiffusive}, lower bound]
    We begin by defining a tessellation $\mathcal{T}_{\kappa n}$ of the box $\Lambda_n$ with disjoint boxes of side length $\lfloor2\kappa n\rfloor$ for $\kappa\in (0, 1)$ (eventually chosen small). Let $\{x_i\}_{i \ge 0}$ be the centers of those boxes so that 
    \begin{equation*}
        \{\Lambda_{\kappa n}(x_i)\}_{i = 1}^{\lfloor\kappa^{-d}\rfloor} \in \mathcal{T}_{\kappa n}.
    \end{equation*}
    We restrict ourselves to boxes that are completely contained inside $\Lambda_n$. Observe that for all $\kappa$ small enough there are at least $\lfloor\kappa^{-d}/2\rfloor$ such boxes. Then, by independence
    \begin{equation*}
        \mathbf{P}\left( \not\exists i,  x_i \in \mathcal{C}^{(\kappa n)}_1(x_i)\right) \le \mathbf{P}\left(0 \not \in \mathcal{C}^{(\kappa n)}_1\right)^{\lfloor\kappa^{-d}/2\rfloor}\le (1 - \theta(\beta, s)/2)^{\lfloor\kappa^{-d}/2\rfloor} \le \varepsilon,
    \end{equation*}
    where in the last step we used that $\kappa>0$ small enough. Let $x^*$ be one arbitrarily chosen $x_i$ such that $x_i \in \mathcal{C}^{(\kappa n)}_1(x_i)$. W.l.o.g. we assume in the remainder of the argument that $x_i = 0$. By Lemma~\ref{lemma:SecondLargest} w.h.p.\ the cluster $\mathcal{C}^{(n)}_1$ coincides with $\mathcal{C}_\infty |_{\Lambda_{n}}$. 

    We now observe that one can couple the walk on $\mathcal{C}_\infty |_{\Lambda_{n}}$, which we denote by $\{Y_k\}_{k \ge 0}$, with the one on $\mathcal{C}_\infty$ which we denote by $\{\widetilde{X}_k\}_{k \ge 0}$ so that the former achieves $\{\widetilde{\tau}_{\kappa n} \le c'_\varepsilon \kappa^2 n^2\}$ with at least the same probability. Indeed, before any edge exiting the box $\Lambda_{\kappa n}$ is met, the walks can be coupled in the trivial way. Then the outgoing edge $e$ is either present in both clusters or only in $\mathcal{C}_\infty$ and not in $\mathcal{C}_\infty |_{\Lambda_{n}}$. In the first case the coupling is trivial, in the second either $\{Y_k\}_{k \ge 0}$ exits immediately the box $\Lambda_{\kappa n}$ through $e$ achieving $\widetilde{\tau}_{\kappa n}$ or it stays inside. In the latter case the walk conditioned to stay inside the box has the same transition probability as $\{\widetilde{X}_k\}_{k \ge 0}$. Thus, the coupling is valid until $\widetilde{\tau}_{\kappa n}$ which is always achieved first by $Y$ (or it is achieved at the same time).
    % We observe that the event $\{\widetilde{\tau}_{\kappa n} \le c'_\varepsilon \kappa^2 n^2\}$ has higher (quenched and hence annealed) probability on the cluster $\mathcal{C}_\infty$ w.r.t.\ the cluster $\mathcal{C}_\infty |_{\Lambda_{n}}$. Indeed, there are more chances to exit $\mathcal{C}_\infty$ and the random walk on $\mathcal{C}_\infty$ and $\mathcal{C}_\infty |_{\Lambda_{n}}$ can perfectly coupled until time $\widetilde{\tau}_{\kappa n}$ (in $\mathcal{C}_\infty$).
    Thus we obtain, applying Corollary~\ref{cor:QuenchedExit},
    \begin{equation*}
        P^\omega_0\left( \widetilde{\tau}_{\kappa n} \le c'_{\varepsilon} \kappa^2 n^2 \right) \le \varepsilon.
    \end{equation*}
    Now, recall that $T_k$ is the $k$-th time of jump of the random walk. By standard estimates for random walks (the estimate being very rough here), for every $\delta > 0$ there exists $\eta= \eta(\delta)>0$ such that for all $k$ large enough
    \begin{equation*}
        P_0^\omega\left( T_{k} < \eta(\delta) k \right) \le \delta.
    \end{equation*}
   Thus, there exists $c_\varepsilon'' > 0$ such that $P_0^\omega\left( T_{c'_\varepsilon n^2} < c''_\varepsilon n^2 \right) \le \varepsilon$. Using Corollary~\ref{cor:QuenchedExit}, we obtain for any $\kappa>0$, 
    \begin{align*}
        P_0^\omega\left( \tau_{\kappa n} \le c_\varepsilon'' \kappa^2 n^2 \right) 
        &\le P_0^\omega\left( \tau_{\kappa n} \le c_\varepsilon'' \kappa^2 n^2,  T_{c'_\varepsilon \kappa^2 n^2} \ge c''_\varepsilon \kappa^2 n^2\right) + P_0^\omega\left( T_{c'_\varepsilon \kappa^2 n^2} < c''_\varepsilon \kappa^2 n^2\right)\\
        &\le P_0^\omega\left( \widetilde{\tau}_{\kappa n} \le c'_\varepsilon \kappa^2 n^2 \right) + \varepsilon\\
        & \le 2\varepsilon.
    \end{align*}
    Now, observe that here $\kappa>0$ here can be chosen arbitrarily small. In particular, we can choose it in such a way that with high probability (see Lemma~\ref{lemma:VolLargeSets}) we have
    \begin{equation*}
        |\pi(\Lambda_{\kappa n} \cap \mathcal{C}^{(n)}_1)| \le 2\bar{d}\kappa^d \le \frac{1}{10},
    \end{equation*}
    where we recall that $\mathbf{E}[\mathrm{deg}^\omega(0)] = \bar{d}>0$.
    
    Observe that this, together with Theorem~\ref{theo:Baumler}, immediately implies w.h.p.\ the existence of a set $A \subseteq ([-n, -\kappa n] \cup [\kappa n, n])^d \cap \mathcal{C}_1^{(n)}$ with $\pi(A) \ge 1/3$ (recall the definition of $\pi(A)$ in Section~\ref{sect:Cond}), such that 
    \begin{equation*}
        E_0^\omega\left[\tau_A \right] \ge (1-2\varepsilon)c_\varepsilon'' \kappa^2 n^2.
    \end{equation*}
    We conclude by directly applying \cite[Theorem~2]{Oliveira}. 
\end{proof}

\section{The truncated model}\label{sect:Trunc}

In this section we work with the model in which all edges above a certain threshold $N>0$ are truncated. In \cite{BergerLRP, BaumlerShapeTheo} it was shown, respectively for the case $s \in (d, d+2)$ and $s \ge d+2$, that it is possible to truncate the LRP and still observe a unique infinite component (resp.\ giant in finite volume) for all $\beta>\beta_c$.

Then, one can use the fact that the truncated model essentially behaves similarly to nearest-neighbour Bernoulli percolation to deduce properties of the LRP. Broadly speaking, we will also employ this strategy. In this section we confirm that the truncated model satisfies the isoperimetric properties of Bernoulli percolation. The arguments here are inspired by \cite{AntalPisztora, PeteIso,MitscheMartinezKiwi, HaalandSarkovic}.

Firstly, we introduce the $N$-truncated percolation model, that is, we will study the percolation model on $\mathbb{Z}^d$ with the modified connection probabilities 
\begin{equation*}
    \mathbf{p}^N(x, y)  \coloneqq \mathbf{p}(x, y)\mathds{1}\{\|x-y\|_\infty\le N\},
\end{equation*}
where $ \mathbf{p}(x, y)$ is as in~\eqref{eqn:ConnectionProbs}.
By the results of \cite{BergerLRP, BaumlerShapeTheo}, there exists $N>0$ such that $\mathbf{p}^N(x, y)$ is supercritical. This justifies the definition of $\mathcal{C}_1^{(n), N}$ to be the giant component of the $N$-truncated percolation inside $\Lambda_n$. It was shown in \cite[Equation~(38)]{BaumlerShapeTheo}  that for some $c>0$,
\begin{equation}\label{eqn:SecondLargTrunc}
    \mathbf{P}^N\left( t \le |\mathcal{C}_0| < \infty \right) \le e^{-ct^{1/d}}.
\end{equation}

Denoting by $E_N(A, B)$ the number of edges between $A$ and $B$ in the truncated model, the main result of this section is the following: 

\begin{prop}\label{prop:IsoTruncated}
     There exist constants $c_{\ell}, c_{b} > 0$ (only depending on $d,N$) such that $\mathbf{P}^N$-almost surely, for all $n$ large enough, we have that all sets $A \subseteq \mathcal{C}_1^{(n), N}$ with $c_{b} \log(n)^{\frac{d}{d-1}} \le |A| \le n/2$ are such that
    \begin{equation*}
        |E_N(A, A^c)| \ge c_{\ell} |A|^{\frac{d-1}{d}}\ge \frac{c_{\ell}}{N^d}|E_N(A, \Lambda_n)|^{(d-1)/d}.
    \end{equation*}
\end{prop}

The remainder of the section is devoted to the preparation of the proof of Proposition~\ref{prop:IsoTruncated} and finally the proof thereof.

To this end, recall that $\Lambda_n$ denotes the box of side length $n$. We define a partition of the box $\Lambda_n$ into tiles of $\infty$-radius $\rho=mN$, where the reader should think of $m=m(N)>0$ as a large integer (say $m=N^c$ for some large $c > 0$). We say that a box of side length $\rho$ has a \textbf{crossing component} if there exists a connected component (internal of $\Lambda_\rho$) that is closer than $m$ from all $(d-1)$-dimensional faces of the cube $\Lambda_\rho$. We say that a tile $\Lambda_\rho$ is \textbf{good} if it has one crossing component.

\begin{prop}\label{prop:TassellationSize}
    For any $\varepsilon>0$ there exists $m_0 = m_0(\varepsilon)>0$ such that for all $m \ge m_0$ and $\rho = m N$,
    \begin{equation*}
        \mathbf{P}^N\left(\Lambda_\rho \text{ is good}\right) \ge 1 - \varepsilon.
    \end{equation*}
\end{prop}

\begin{proof}
The proof follows the ideas of 
    Theorem 2 of Penrose and Pisztora~\cite{PenrosePisztora} with small modifications. First, call a box of side length $m$ \emph{useful} (inside the larger box of side length $\rho$, and $m$ chosen in such a way that the boxes of side length $m$ perfectly tessellate the box of side length $\rho$) if it has exactly one  component of size at least $\theta m^d/2$ (where $\theta > 0$ is the fraction of vertices in the infinite component). 
    
    By \cite[Theorem 1.8]{BaumlerShapeTheo}, the probability of having at least one such component can be made arbitrarily close to $1$ by choosing $m \ge N$ large enough. Moreover, since only vertices at distance at most $N$ from the boundary of the box may contribute to the vertex boundary of the box, the second largest component is smaller than $\delta m^d$ for arbitrarily small $\delta > 0$ (by choosing $m$ large enough, see for example Theorem 1.2 of~\cite{EasoHutchcroft}). Hence, for any $\varepsilon > 0$, the probability that a box of side length $m$ is useful can be made at least $1-\varepsilon$. Furthermore, by ergodicity, shifting such a box of side length $m$ by a small constant $c=c(\theta)$ in one of the $d$ dimensions, the shifted box is also useful with probability at least $1-\varepsilon$. In particular, if two adjacent boxes of side length $m$ in the box of side length $\rho$ are both useful and all boxes shifted by $c$ between the boxes are also useful, then the largest components of the two adjacent boxes must merge (the overlap is chosen large enough). 
    
    By Theorem 0.0 of~\cite{liggettdomination}, the resulting $k$-dependent percolation (with $k$ only depending on $c$, thus only depending on $\theta$) on boxes of side length $m$ dominates supercritical site Bernoulli percolation. By Proposition 1 of~\cite{PenrosePisztora}, the box of side length $\rho$ thus has one crossing component (in the original graph) with probability $1-\varepsilon'$ for arbitrarily small $\varepsilon' > 0$. %By the same proposition with the same probability, this component will also be part of the largest component when extending this to the whole graph.
Thus, a tile of side length $\rho$ is good with probability at least $1-\varepsilon$, and the lemma follows.
    %Finally, if in the box of side length $\rho$ a second component of Euclidean diameter at least $\rho/5$  existed, such a component would have to avoid all crossing components of all at least $N/6$ boxes of side length $m$ in between, which happens, by a union bound over all $\rho^d$ starting points, with probability at most $\rho^d (1-\varepsilon)^{N/6} \le \varepsilon$ (the latter inequality holds since somewhere the crossings become close, and where we used the fact that $N \ge \log^2(\rho)$) \dmc{I think this is the only point where we need that the tile is not too big (and we cannot take the whole space in particular already here)}. Thus, each box of side length $\rho$ is good with probability at least $1-\varepsilon$.
\end{proof}

We obtain then the following proposition: 
\begin{prop}\label{prop:Volume}
    For all $\eta < \theta(s, \beta)$ there exists $c>0$ such that for any $n$ large enough,
    \begin{equation*}
        \mathbf{P}^N\left(|\mathcal{C}^{(n), N}_1 | \le \eta n^d \right) \le ce^{- n^{d-1}}.
    \end{equation*}
    Moreover, under the obvious coupling between $\mathbf{P}$ and $\mathbf{P}^N$, with $\mathbf{P}$-probability larger than $1 - e^{-cn^{d-1}}$ we have
    \begin{equation*}
        \mathcal{C}^{(n), N}_1 \subseteq \mathcal{C}^{(n)}_1.
    \end{equation*}
\end{prop}
\begin{proof}
  Tessellate the cube of side length $n$ into cells of side length $\rho$ (assuming w.l.o.g.\ that this is possible), that are good with probability at least $1-\varepsilon$ by Proposition~\ref{prop:TassellationSize}. Moreover, if two adjacent boxes of side length $\rho$ as well as the box of side length $\rho$ shifted half way between them are all good, then the giant components of the two adjacent boxes must merge (we used $m=N^c \gg N$). Thus, again by Theorem 0.0 of~\cite{liggettdomination}, the resulting $k$-dependent percolation dominates supercritical site percolation and thus, by Proposition 1 of~\cite{AntalPisztora}, a giant component of size at least $\mu n^d$ with probability at least $\exp(-cn^{d-1})$. Moreover, with the same probability this giant is contained in the infinite component, and the proposition follows. 
\end{proof}

We are now ready for the proof of Proposition~\ref{prop:IsoTruncated}: 

\textit{Proof of Proposition~\ref{prop:IsoTruncated}:} Tessellate as before into boxes of side length $\rho$ that are good with probability at least $1-\varepsilon$ as seen in Proposition~\ref{prop:TassellationSize}. Define an auxiliary graph by introducing one vertex for each such box, and if two adjacent boxes as well as the shifted box half way between them are all good, then draw an edge between the corresponding vertices. Once again, if two adjacent boxes of side length $\rho$ as well as the box of side length $\rho$ shifted half way between them are all good, then the giant components of the two adjacent boxes must merge (we used $m=N^c \gg N$). 

Thus, again by Theorem 0.0 of~\cite{liggettdomination}, the resulting $k$-dependent percolation dominates supercritical site percolation. By Corollary 1.4 of~\cite{PeteIso}, with probability $1-\varepsilon$, for any $\varepsilon > 0$, in the auxiliary graph, every set $S$ inside the giant component of this graph with $c\log(n)^{d/(d-1)} \le |S| \le |\mathcal{H}_1^{(n),N}|/2$ (with $\mathcal{H}_1^{(n), N}$ denoting the largest component of the auxiliary graph), satisfies $|E(S,S^c)| \ge \alpha|S|^{1-1/d}$ for some $\alpha > 0$. Now, consider a subset $A \subseteq \mathcal{C}_1^{(n),N}$ of the original graph with $c_{b}\log(n)^{d/(d-1)} \le |A| \le n/2$ and $c_{b} > 0$ large enough. Since the lower bound on isoperimetry is attained for connected sets, we may $A$ assume to be connected. Moreover, since the number of vertices in each box is at most $\rho^d$, $A$ must be contained in at least $|A|/\rho^d$ many boxes (and in at most $|A|$ many). Now, consider the set $A'\supseteq A$ as the filled-up version of $A$, which is the set that contains all vertices in all boxes that have at least one vertex of $A$. Note that $|A| \le |A'| \le \rho^d |A|$, and that the set of boxes corresponding to $A$ (or $A'$) is diagonally connected. 

Suppose first that at least half of the boxes of $A$ do not belong to the giant component of the auxiliary graph. Since each such component contains at most $c'(\log n)^{d/(d-1)}$ vertices for some $c' > 0$, we get $|E(A,A^c)|\ge c''|A|^{1-o(1))} \ge c_{\ell}|A|^{\frac{d-1}{d}}$, and we are done. Otherwise, if $A \ge c_{b}\log(n)^{d/(d-1)}$ and $A'$ is diagonally connected, at least a constant fraction $c_1$ of the boxes of $A'$ must belong to the giant component of the auxiliary graph. Let $A''$ be the restriction of $A'$ to boxes in this giant component, and note that $|A''| \ge c'\log(n)^{d/(d-1)}$. For this set, with $S$ denoting the corresponding boxes of $A''$, by the above, in the auxiliary graph, $|E(S,S^c)| \ge \alpha |S|^{1-1/d}$. By adapting constants, this also translates to $A''$ (and thus to $A$), since a boundary edge in the auxiliary graph translates to a boundary edge in the original graph: indeed, if there is an edge between a box $\Lambda \in S$, and $\Lambda' \in S^c$, then there exists $v(\Lambda) \in A$ contained in the box $\Lambda$, whereas no vertex in $\Lambda'$ belongs to $A$. Since the edge in the auxiliary graph is there, the giant components of $\Lambda$ and $\Lambda'$ merge, so there is a direct path from vertices in the giant inside $\Lambda$ to the one inside $\Lambda'$, without passing through other boxes. In particular, there is a path from $v(\Lambda)$ to $\Lambda'$, and therefore there is a boundary edge.

The second inequality follows trivially from the first one since each vertex of $A$ has at least degree $1$. \hfill\qed

\section{Preparation for the upper bound}\label{sect:Prep}

The main result of this section will be called the \emph{allocation theorem}. We split the proof into the two regimes $s \in (d,2d)$ and $s \ge 2d$ since the scales we will use in the two regimes are different. The case $s \in (d, 2d)$ is heavily inspired by the proof of~\cite[Theorem~2.1]{CrawfordSly1}. However, since therein the result is stated for $s \in (d, d+2)$ only and since our proof is simpler (as we do not require the spectral gap estimates derived there), we will nevertheless give a proof of the first regime as well.

\begin{thm}\label{theo:Partition}
    Let $s>d$. For any $\varepsilon>0$, $n$ large enough and $M \ge n^\varepsilon$, there exist deterministic constants $\rho_1, \rho_2, c_1, \Delta_1 > 0$ and an event $A_n$ of $\mathbf{P}$-probability at least $1-\exp(-c_1\log^2(n))$ such that on $A_n$ the following holds.
    
    There exists a function $\phi \colon \mathcal{C}_1^{(n)} \setminus \mathcal{C}_1^{(n), N} \to \mathcal{T}_M$ with $x \mapsto \Lambda_M$ such that denoting
    \begin{equation*}
        \widetilde{\mathcal{C}}^{(n)}(\Lambda_M) \coloneqq \mathcal{C}_1^{(n), N}(\Lambda_M) \cup \{x \colon \phi(x) = \Lambda_M\},
    \end{equation*}
    we have:
    \begin{enumerate}
        \item $|\mathcal{C}_1^{(n), N}(\Lambda_M) \bigcap \Lambda_M| \ge \rho_1 M^{d}$ for all $\Lambda_M \in \mathcal{T}_M$.
        \item $|E(\widetilde{\mathcal{C}}^{(n)}(\Lambda_M), \Lambda_n)| \le \rho_2 M^d$ for all $\Lambda_M \in \mathcal{T}_M$ and $\widetilde{\mathcal{C}}^{(n)}(\Lambda_M)$ is connected.
        \item For all $x \in \mathcal{C}_1^{(n)}\setminus \mathcal{C}_1^{(n), N}$ we have 
        \[
        d_{\widetilde{\mathcal{C}}^{(n)}(\Lambda_M)}(x, \phi(x)) \le \log^{\Delta_1}(n).
        \]
    \end{enumerate}
    We also have that $|\mathcal{C}_2^{(n)}(\Lambda_n)| \le \log^{\Delta_{1}}(n)$. Furthermore, with $\mathrm{diam}(\cdot)$ denoting the diameter (in graph distance), for $s \in (d, 2d)$ we have that $\mathrm{diam}(\mathcal{C}^{(n)}_1)\le \log^{\Delta_{1}}(n)$.
\end{thm}

% \begin{thm}\label{theo:Partition}
%     Let $s>d$. For any $\varepsilon>0$, $n$ and $M \ge n^\varepsilon$ large enough, there exist universal constants $\rho_1, \rho_2, C_1, c_2, \Delta_1$ such that there exists an event $A_n$ of $\mathbf{P}$-probability at least $1-C_1 \exp(-c_2\log^2(n))$ such that on this event there exists a partition of $\mathcal{C}_1^{(n)}$ into connected components $\{\widetilde{\mathcal{C}}^{(n)}(\Lambda_M)\}_{\Lambda_M \in \mathcal{T}_M}$ with the properties
%     \begin{enumerate}
%         \item $|\widetilde{\mathcal{C}}^{(n)}(\Lambda_M) \bigcap \Lambda_M| \ge \rho_1 M^{d}$ for all $\Lambda_M \in \mathcal{T}_M$.
%         \item $|\widetilde{\mathcal{C}}^{(n)}(\Lambda_M)| \le \rho_2 M^d$ \dm{for all $\Lambda_M \in \mathcal{T}_M$}.
%         \item \dmc{why is this a third item? It is not related to the partition. I'd put it not as item}For all $x \in \mathcal{C}_1^{(n)}$ we have $d(x, \mathcal{C}_1^{(n), N}) \le \log^{\Delta_1}(n)$.
%     \end{enumerate}
%     Furthermore, for $s \in (d, 2d)$ we have that $\mathrm{diam}(\mathcal{C}^{(n)}_1)\le \log^{\Delta_{1}}(n)$, where $\mathrm{diam}(\cdot)$ denotes the diameter in graph distance.
% \end{thm}

\subsection{The allocation theorem for $s \in (d, 2d)$}

In this subsection we prove Theorem~\ref{theo:Partition} in the case $s \in (d, 2d)$. Observe that the third enumerated item in Theorem~\ref{theo:Partition} is clearly implied by the subsequent estimate $\mathrm{diam}(\mathcal{C}^{(n)}_1)\le \log^{\Delta_{1}}(n)$, and we thus will prove the latter.

\subsubsection{Preparation and some notation}

We fix $\zeta_1>0$ large enough and the scales $n \ge M(n) \eqqcolon M_1 \ge 1$ and
\begin{equation}\label{eqn:Scales}
    M_2 \coloneqq \lfloor M^{(s-d)/d}\log^3(M)\rfloor, \qquad M_3 \coloneqq \lfloor M_2^{1/2} \rfloor, \qquad \qquad M_4 \coloneqq \lfloor \log^{\zeta_1}(n) \rfloor. 
\end{equation}
 We will choose the scales so that $M_1 \gg M_2 \gg M_3 \gg M_4 = \log^{\zeta_1}(n)$. For any pair of integers $L>0$ and $K>0$ with $K$ being a divisor of $L$, recall that $\mathcal{T}_K(\Lambda_L)$ is the partition of the box $\Lambda_L$ of side length $L$ into boxes of side length $K$. For $j = 1 \dots, 3$ we define successively the nested partitions
\begin{equation*}
    \mathcal{T}_{M_{j+1}} \coloneqq \bigcup_{\Lambda_{M_j} \in \mathcal{T}_{M_j}} \mathcal{T}_{M_{j+1}}(\Lambda_{M_j}).
\end{equation*}
Let $\mathcal{F}_*$ be the $\sigma-$field that includes all edges on pairs of vertices $x, y$ contained in the same $\Lambda_{M_4} \in \mathcal{T}_{M_4}$. We introduce the $\sigma$-fields
\begin{equation}\label{eqn:SigmaH}
    \mathcal{H}_j \coloneqq \sigma\left( \mathcal{F}_* \vee \sigma(\omega_{xy} \colon \|x-y\|_\infty < M_j)  \right).
\end{equation}
Let $\mathscr{H}_j$ be the graph measurable w.r.t.\ $\mathcal{H}_j$.

\subsubsection{Construction of the core}

We define the core of the graph by introducing random sets $\mathcal{G}(\Lambda_{M_4}) \subseteq \Lambda_{M_4}$ as the largest connected component $\mathcal{C}(\Lambda_{M_4})$. Clearly, we are revealing only edges of length at most $\sqrt{d}M_4$ (and only within boxes of the partition $\mathcal{T}_{M_4}$) in this first step of the procedure. By union bound and Proposition~\ref{prop:Volume} one obtains that there exists an event $\mathcal{O}$ with $\mathbf{P}$-probability larger than $1 - \exp(-\log^2(n))$ on which $|\mathcal{C}_1^{(n), N}(\Lambda_M)| \ge \rho_1 M_4^d$ for all $\Lambda_{M_4} \in \mathcal{T}_{M_4}$. Furthermore, on the same event $\mathcal{O}$, we have that $\mathcal{C}_1^{(n), N}(\Lambda_M) \subseteq \mathcal{G}(\Lambda_{M_4})$ and hence $|\mathcal{G}(\Lambda_{M_4})| \ge \rho_1 M_4^d$ in every $\Lambda_{M_4} \in \mathcal{T}_{M_4}(\Lambda_n)$.

We observe that the allocation function $\phi$ implicitely defined here is the trivial one (each site in the cluster is assigned to its own box). Further notice that the graph distance between a vertex $x \in \Lambda_M$ and $\mathcal{C}_1^{(n), N}(\Lambda_M)$ is deterministically at most $M_4^d \le \log^{\Delta_1}(n)$.

Let us consider the graph $\mathcal{G}$ with vertex set $\bigcap_{\Lambda_{M_4} \in \mathcal{T}_{M_4}} \mathcal{G}(\Lambda_{M_4})$ and then revealing all edges between those vertices. We prove the following result, which is similar to \cite[Lemma~2.6]{CrawfordSly1}:

\begin{prop}\label{prop:CoreProps}
    There exist constants $\Delta_{2}, \rho_1, \rho_2, c_2 > 0$ such that for all $n$ large enough there exists an event $\mathcal{O}$ with $\mathbf{P}$-probability larger than $1 - \exp(-c_2 \log^2(n))$ such that on $\mathcal{O}$,
    \begin{enumerate}
    \item $\rho_1 M_4^d \le |\mathcal{G}(\Lambda_{M_4})|- 1 \le |E(\mathcal{G}(\Lambda_{M_4}), \Lambda_n)|  \le \rho_2 M_4^d $ in every $\Lambda_{M_4} \in \mathcal{T}_{N_4}(\Lambda_n)$.
        \item For all $j = \{1, 2, 3, 4\} $, the graph $\mathcal{G}(\Lambda_{M_j}) \subseteq \Lambda_{M_j}$ is connected within $\Lambda_{M_j}$ for all $\Lambda_{M_j} \in \mathcal{T}_{M_j}(\Lambda_n)$. 
        \item All $\mathcal{G}(\Lambda_{M_j})$ are connected by an edge to $\mathcal{G}(\Lambda_{M_j}')$ whenever $\Lambda_{M_j}$ and $\Lambda_{M_j}'$ are $*$-adjacent. 
        \item For $j = 2, 3$, $\mathrm{diam}(\mathcal{G}(\Lambda_{M_j})) \le \log^{\Delta_{2}}(M_j) M_4$.
    \end{enumerate}
\end{prop}
\begin{proof}
    The first property follows by Proposition~\ref{prop:Volume} and Lemma~\ref{lemma:VolLargeSets}. The second property follows by construction. The third property follows from the fact that the probability that two giant clusters in neighbouring boxes are not connected by an edge is at most
    \begin{equation*}
        \exp\left( - c \frac{\rho_1^{2}M_4^{2d}}{M_4^s}\right),
    \end{equation*}
    which can be made less than $\exp\{-c\log^2(n)\}$ by choosing $\zeta_1>0$ large enough (recall that $c>0$ is adjusted step by step).
    The fourth estimate is \cite[Lemma~2.7]{CrawfordSly1}, which also yields that the largest connected component not connected to the core is at most of size $M_4\log^2(n)$.
\end{proof}

\subsubsection{Allocating small clusters to the core}\label{sect:SmallToCoresSmall}

Recall the definition of the $\sigma$-field \eqref{eqn:SigmaH} and that $\mathcal{F}_*$ is the $\sigma$-field associated to the core. For $j = 1, \dots, 4$ we define
\begin{equation}\label{eqn:SigmaF}
    \mathcal{F}_j \coloneqq \sigma\left( \mathcal{F}_* \vee \sigma(\omega_{xy} \colon \|x-y\|_\infty < M_j \wedge M_2), x, y \in \cup_{\Lambda_{M_4} \subseteq \Lambda_{M_j}} \mathcal{G}(\Lambda_{M_4}), \Lambda_{M_j} \in \mathcal{T}_{M_j} \right).
\end{equation}

The allocation of small clusters in $\Lambda_n \setminus \mathcal{G}$ to $\mathcal{G}(\Lambda_{M_j}), j = 3, 2, 1$ will be done in two steps. Firstly, we will introduce edges at smaller scales to then move towards longer edges in the second step.

Let $\mathscr{H}_j$ be the graph measurable w.r.t.\ $\mathcal{H}_j$. Let $\mathcal{C}(j, \mathcal{G})$ be the connected component of $\mathcal{G}$ once we have revealed all edges of length at most $M_j$ with the properties showed in Proposition~\ref{prop:CoreProps}. Consider the sets
\begin{equation*}
    R_k(\mathcal{G}) = \{x \in \mathcal{C}(j, \mathcal{G}) \colon d_{\mathscr{H}_j} (x, \mathcal{G}) \le k\}.
\end{equation*}
Let $A_0(\Lambda_{M_2}) = \mathcal{C}(j, \mathcal{G})$ and then define inductively
\begin{equation*}
    A_{k+1}(\Lambda_{M_2}) = \left\{x \in R_{k+1}(\mathcal{G}) \setminus R_k(\mathcal{G}) \colon d_{\mathscr{H}_j}(x,A_{k}(\Lambda_{M_2})) \le 1 \right\},
\end{equation*}
with a pre-established deterministic rule to break ties between distinct boxes $\Lambda_{M_2}$ and $\Lambda'_{M_2}$ so that the sets $\{A_{k+1}(\Lambda_{M_2})\}_{\Lambda_{M_2}}$ are disjoint. Define the assignment function
\begin{equation}\label{eqn:Allocation}
    \phi_j (x) \coloneqq \Lambda_{M_2} \qquad \text{if} \qquad x \in \bigcup_{k}A_{k}(\Lambda_{M_2}).
\end{equation}
Then, defining
\begin{equation*}
    \mathcal{G}(j, \Lambda_{M_2}) \coloneqq (\mathcal{G} \cap \Lambda_{M_2}) \bigcup\{x \colon \phi_j (x) = \Lambda_{M_2} \},
\end{equation*}
we see that the sets $\{\mathcal{G}(j, \Lambda_{M_2})\}_{\Lambda_{M_2} \in \mathcal{T}_{M_2}}$ partition $\mathcal{C}(j, \mathcal{G})$ into connected subsets.

We summarise the properties of the sets $\{\mathcal{G}(j, \Lambda_{M_2})\}_{\Lambda_{M_2} \in \mathcal{T}_{M_2}}$ in the following proposition.

\begin{prop}\label{prop:IntermediateAllocationSmalls}
    There exist constants $\rho_1, \rho_2>0$, $\Delta_{3}, c_3 > 0$ such that there exists an event $\mathcal{U}$ with $\mathbf{P}$-probability larger than $1 - \exp(-c_3 \log^2(n))$ such that on $\mathcal{U}$ for all $\{\mathcal{G}(3, \Lambda_{M_2})\}_{\Lambda_{M_2} \in \mathcal{T}_{M_2}}$,
    \begin{enumerate}
        \item $\rho_1 M_2^d \le |\mathcal{G}(3, \Lambda_{M_2})| - 1 \le E(\mathcal{G}(3, \Lambda_{M_2}), \Lambda_n)  \le \rho_2 M_2^d $ in every $\Lambda_{M_2} \in \mathcal{T}_{M_2}(\Lambda_n)$.
        \item $\mathrm{diam}(\mathcal{G}(3, \Lambda_{M_2})) \le \log^{\Delta_{3}}(M_2)$ for all $\Lambda_{M_2} \in \mathcal{T}_{M_2}(\Lambda_n)$.
        \item The largest cluster not connected to the $\mathcal{G}(3, \Lambda_{n})$ is of size $M_4^{s-d}\log^4(n)$.
    \end{enumerate}
\end{prop}
\begin{proof}
Before proceeding let us set up a useful computation. Let $x \in \Lambda_n$; we compute the probability that $x$ connects by an edge to the core $\Lambda_m(x)^c \cap \mathcal{G}_{m}$ when all edges of length at least $m$ have not been revealed yet and $\mathcal{G}_{m}$ has positive density $\eta>0$ in each box. By only looking at neighbouring boxes, there exist constants $c, c' > 0$ such that for $m>0$ large,
\begin{align}
    \mathbf{P}\left( \exists y \in \Lambda_m(x)^c \cap \mathcal{G}_{m} \colon \mathds{1}\{\omega(y, x) = 1\}\right) & = 1 - \prod_{y \in \Lambda_m(x)^c \cap \mathcal{G}_{m}} (1 - \mathbf{p}(x, y)) \nonumber \\
    & \ge 1 - \exp\left( - \sum_{y \in \Lambda_m(x)^c \cap \mathcal{G}_{m}} \mathbf{p}(x, y) \right)\nonumber\\
    & \ge 1 - \exp\left( - c |\Lambda_m| \frac{\beta}{m^{s}} \right)\nonumber\\
    & \ge 1  - \exp\left(- c' \frac{m^d}{m^s}\right).\label{eqn:NonConn}
\end{align}
We will employ this computation (for different $m>0$) repeatedly in this proof and beyond.

The lower bound for the volume follows directly from the construction of the core. Note that the probability that a vertex connects to the core by an edge when revealing all edges longer than $M_4$ is, for any $x$,
    \begin{equation*}
        \mathbf{P}\left( \exists y \in \Lambda_{M_4}(x)^c \cap \mathcal{G}_{M_4} \colon \mathds{1}\{\omega(y, x) = 1\}\right) \ge 1 - e^{-c M_4^{d-s}}.
    \end{equation*}
    Hence, by independence of the edges, the probability that a cluster $\mathscr{C}$ does not connect to the core by an edge is bounded from above, for $n$ large, by
    \begin{equation*}
        \mathbf{P}\left( \not\exists y \in \mathcal{C}^c \cap \mathcal{G}_{M_4}, x \in \mathcal{C} \colon \mathds{1}\{\omega(y, x) = 1\}\right) \le e^{-c M_4^{d-s} |\mathscr{C}|}.
    \end{equation*}
    Since this holds for every vertex, with probability at least $1-e^{-c\log^2 (n)}$ any cluster outside the core is at most of size $O(M_4^{s-d}\log^2(n))$. Hence, the number of steps (in graph distance) that a vertex $x$ takes to connect to the core is at most $M_4^{s-d}\log^2(n)$ with probability, for $n$ large, at least
    \begin{equation*}
        1 - e^{-c\log^2(n)}.
    \end{equation*}
    By Proposition~\ref{prop:CoreProps}, distances within the core are with probability larger than $1 - e^{-c\log^2(n)}$ bounded from above uniformly by $\log^{\Delta_{2}}(M_j) M_4$, yielding the required diameter bound. Together with the fact that sites outside the core connect to it within $M_4^{s-d}\log^2(n)$ steps, we obtain the required diameter bound for the graph after adding longer edges. Observe that these observations together with Lemma~\ref{lem:MaxDeg} imply that the second largest cluster must be smaller than $M_4^{s-d}\log^4(n)$.
    
    For the volume upper bound we proceed similarly. By Lemma~\ref{lem:MaxDeg} the largest degree is bounded from above by $U\log(n) \le \log^2(n)$. Hence, the volume of the sets added to $\Lambda_{M_2}$ is bounded from above by $\log^{\Delta_{1}}(M_j) M_4$. Note that the Euclidean distance between any $x \in \mathcal{G}(3, \Lambda_{M_2})$ and the corresponding $\Lambda_{M_2}$ is at most 
    \begin{equation*}
        M_3M_4^{d-s}\log^2(n).
    \end{equation*}
    We conclude that 
    \begin{equation*}
         |E(\mathcal{G}(3, \Lambda_{M_2}), \Lambda_n)| \le \rho_2 M_2^d + \rho_2 M_2^{d-1} M_3 M_4^{d-s}\log^2(n).
    \end{equation*}
    Choosing the scales so that $M_2 \ge M_3 M_4^{d-s}\log^2(n)$ yields the result by adjusting $\rho_2$.
\end{proof}

We have so far revealed all edges of $\ell_{\infty}$ length at most $M_3$. We now reveal all remaining edges and we need to focus on the small clusters $\mathscr{C}_j, j \in \mathscr{S}$, where $\mathscr{S}$ denotes the finite index set of small clusters once we reveal all edges in $\mathcal{H}_{3}$.

 We immediately observe that, on the event $\mathcal{U}$ of Proposition~\ref{prop:IntermediateAllocationSmalls},
\begin{itemize}
    \item $|\mathscr{C}_j| \le M_4^{s-d} \log^4(n)$ for all $j \in \mathscr{S}$.
    \item $|\mathscr{S}|\le n^d$.
\end{itemize}
We will focus on $\mathscr{C}_j$, $j \in \mathscr{S}$, that connect to $\mathcal{C}(3, \mathcal{G})$. These will be controlled through a comparison to a subcritical branching process. We define an allocation from the set of small clusters $\mathscr{C}_j$, $j \in \mathscr{S}$, to $\mathcal{T}_{M_2}$ (called $\phi$ analogously to \eqref{eqn:Allocation}) and then obtain
\begin{equation*}
    \mathcal{G}(0, \Lambda_{M_2}) \coloneqq \mathcal{G} (3, \Lambda_{M_2}) \bigcup\{x \colon \phi(x) = \Lambda_{M_2} \}.
\end{equation*}
Observe that, once again, the sets $\{\mathcal{G}(0, \Lambda_{M_2})\}_{\Lambda_{M_2} \in \mathcal{T}_{M_2}}$ partition the largest component inside $\mathcal{C}_1^{(n)}$ (now completely revealed).

We summarise the properties of the sets $\{\mathcal{G}(0, \Lambda_{M_2})\}_{\Lambda_{M_2} \in \mathcal{T}_{M_2}}$ in the following proposition:

\begin{prop}\label{prop:BranchingSmalls}
    There exist constants $\rho_1, \rho_2 > 0$ and $c_4>0$ such that for all $n$ large enough there exists an event $\mathcal{B}$ on which the partition $\{\mathcal{G}(0, \Lambda_{M_2})\}_{\Lambda_{M_2} \in \mathcal{T}_{M_2}}$ of $\mathcal{C}_1^{(n)}$ into connected components has the property 
    \begin{equation*}
        \rho_1 M_2^d \le |\mathcal{G}(0, \Lambda_{M_2})| - 1\le|E(\mathcal{G}(0, \Lambda_{M_2}), \Lambda_m| \le \rho_2 M_2^d.
    \end{equation*}
    Also,
    \begin{equation*}
        \mathrm{diam}\left(\mathcal{G}(0, \Lambda_{M_2})\right) \le M_4^{s-d}\log^6(n) + \mathrm{diam}\left(\mathcal{G}(3, \Lambda_{M_2}) \right).
    \end{equation*}
    Moreover,
    \begin{equation*}
        \mathbf{P}\left( \mathcal{B} \mid \mathcal{U}, \mathcal{H}_{M_3} \right) \ge 1 - e^{-c_4 \log^2(n)}.
    \end{equation*}
\end{prop}

\begin{proof}
    We will compare the number of $\mathscr{C}_j, j =1, \dots, m$ added to the already existing sets $\mathcal{G}(3, \Lambda_{M_2})$ to a subcritical branching process.

    Let 
    \begin{equation*}
        \hat{\mathrm{deg}}(\mathscr{C}) = |\{\{x, y\} \colon x \in \mathscr{C}, y \in \mathscr{C}', \|x - y\|_\infty \ge M_3 \text{ for some }\mathscr{C}' \neq \mathscr{C} \}|.
    \end{equation*}
    For each $\mathscr{C}_j \in \mathscr{S}$, 
    \begin{equation*}
        \mathbf{P}\left( \hat{\mathrm{deg}}(\mathscr{C}) = k \mid \mathcal{U}, \mathcal{H}_{M_3}\right) \le \frac{|\mathscr{C}|^k}{k!} M_{3}^{k(d-s)}.
    \end{equation*}
    We now analyse the connected components of the graph that records the connections between $\mathscr{C}_j, \mathscr{C}_i  \in \mathscr{S}$. Since we are only revealing edges longer than $M_3$, by the observations made a priori on $\mathscr{C}_j, j \in \mathscr{S},$ we can compare them to a branching process where the offspring distribution has mean 
    \begin{equation*}
        \eta_n = M_4^{s-d}\log^2(n)M_3^{d-s}.
    \end{equation*}
    By the choice of $M_3$ made in \eqref{eqn:Scales} we can bound from above $1> \eta > \eta_n$ uniformly. Hence, the branching process is subcritical and, a fortiori, we obtain for some constant $c>0$ and $n$ large
    \begin{equation*}
        \mathbf{P}\left(|\tilde{\mathcal{C}}(\mathscr{C})| > \log^2(n)\mid \mathcal{U}, \mathcal{H}_{M_3} \right) \le e^{-c\log^2(n)},
    \end{equation*}
    where $\tilde{\mathcal{C}}(\mathscr{C})$ denotes the connected component of $\mathscr{C} \in \mathscr{S}$ (where we join small clusters through long edges). Taking a union bound over all at most $n^d$ clusters, for $n$ large enough 
    \begin{equation*}
        \mathbf{P}\left(\text{there exists } \mathscr{C} \in \mathscr{S}\text{ s.t. }|\tilde{\mathcal{C}}(\mathscr{C})| > \log^2(n)\mid \mathcal{U}, \mathcal{H}_{M_3} \right)\le e^{-c\log^2(n)}.
    \end{equation*}
    We observe that this implies that the number of steps it takes a vertex $x \in \mathscr{C}$ to reach the $\mathcal{G}(0, \Lambda_{M_2})$ it is associated to is bounded from above by $\sup_{\mathscr{S}}|\mathscr{C}| \log^2(n)$, and this yields the required result.
    
    We now focus on the volume. Since the maximal degree of a vertex is uniformly bounded from above by $\log^2(n)$ (cfr.\ Lemma~\ref{lem:MaxDeg}), the volume of the connected components of vertices that were in some $\mathscr{C} \in \mathscr{S}$ before revealing long edges are bounded, with the required probability, by 
    \begin{equation*}
        M_4^{s-d} \log^8(n).
    \end{equation*}
    We are only left with the task of excluding that too many small clusters are allocated to the same box $\Lambda_{M_2}$. Observe that in the procedure above we still have to reveal the edges between the clusters $\mathscr{C}$ and $\mathcal{G}(3, \Lambda_{M_3})$.

    We will require a large deviation estimate on the number of long edges emanating from a box, and then conclude by taking a union bound over the number of boxes.

    Let
    \begin{equation*}
        Y \equiv Y_{\mathcal{G}(3, \Lambda_{M_2})} \coloneq \sum_{\substack{x \in \mathcal{G}(3, \Lambda_{M_3})\\y \in \mathcal{G}^c, \|x-y\|_\infty \ge M_3}} \mathds{1}\{\omega(\{x, y\}) = 1\},
    \end{equation*}
    be the random variable encoding the number of long edges attached to the giant component $\mathcal{G}(3, \Lambda_{M_2})$ inside $\Lambda_{M_2}$.

    Note that for all $x, y$ as above with $\|x-y\|_\infty = r$ we have that $\mathds{1}\{\omega(x, y) = 1\} \sim \mathrm{Ber}(1-\exp{\{-\beta r^{-s}}\}) \approx \mathrm{Ber}(\beta r^{-s}\})$. These are all independent of each other and conditionally independent of all edges we already revealed; hence their centred sum is a martingale with bounded increments. We denote by $\mu_Y$ and $\sigma_Y$ the mean and standard deviation of $Y$ conditional on $\mathcal{U}, \mathcal{H}_{M_3}$. Applying a standard exponential concentration inequality for martingales with bounded increments we can write for any $\lambda, t >0$,
    \begin{equation}\label{eqn:ExpMartingale}
        \mathbf{P}\left(|Y - \mu_Y| > M_2^\lambda \sigma_y\right) \le e^{-tM_2^\lambda}e^{t^2/2}.
    \end{equation}
    Since
    \begin{equation*}
        \mu_Y \le C_1 M_3^{d-s}M_2^d,
    \end{equation*}
    by a trivial bound for Bernoulli random variables,
    \begin{equation*}
        \sigma^2_Y \le C_2 M_3^{d-s}M_2^d.
    \end{equation*}
    Inserting these values into \eqref{eqn:ExpMartingale}, optimising over $t>0$ and choosing $\lambda>1$ so that $M_2^{2\lambda} = M_3^{d-s}M_2^d$, we obtain
    \begin{equation}
        \mathbf{P}\left(|Y - \mu_Y| > M_2^\lambda \sigma_y\right) \le e^{-M_2^{2\lambda}/2}.
    \end{equation}
    Taking a union bound over $\mathcal{T}_{M_2}$ we obtain that
    \begin{equation*}
        \mathbf{P}\left( \exists \Lambda_{M_2} \colon Y_{\mathcal{G}(3, \Lambda_{M_2})} \ge  2(C_1+C_2) M_3^{d-s}M_2^d \right) \le \left(\frac{M}{M_2}\right)^d e^{-M_2^{2\lambda}/2}.
    \end{equation*}
   Recalling  that $M_2 = \log^{\xi_2}(n)$, for $\xi_2>0$ chosen large enough, we obtain for $n$ large
    \begin{equation*}
        \mathbf{P}\left( \exists \Lambda_{M_2} \colon Y_{\mathcal{G}(3, \Lambda_{M_2})}\ge  2(C_1+C_2) M_3^{d-s}M_2^d \right)\le e^{-c\log^2(n)}.
    \end{equation*}
    Finally, with probability larger than $1 - \exp(-c\log^2(n))$ we showed that the total number of vertices allocated to any $\mathcal{G}(3, \Lambda_{M_2})$ in this final allocation step is less than
    \begin{equation*}
        2(C_1+C_2)M_2^d  M_3^{d-s}M_4^{d-2}\log^6(n) \le \rho_2M_2^d,
    \end{equation*}
    as $s>d$ and $M_3$ is chosen large w.r.t.\ $M_4$. This concludes the proof.
\end{proof}

\subsection{The allocation theorem for $s \ge 2d$}

\subsubsection{Construction of the core}

We will define successive refinements of the allocation procedure as in the case $d<s<2d$. The main difference is that we need to use different scales, and that the diameter bound is not true in this regime. In some sense, the allocation is more natural in this regime, as the geometry induced by the graph is much closer to the Euclidean geometry (particularly for $s>2d$, see \cite{BaumlerShapeTheo}). The crucial input from Section \ref{sect:Trunc} is the fact that the giant component occupies a positive density of the sites in large enough boxes.

Let $0<\xi_1 < \xi_2$ be two positive constants whose precise value will be set along the proof. We define $n \ge M(n) = M_1 \ge 1$, and 
\begin{equation*}
    M_2 \coloneqq \lfloor \log^{\xi_2}(n) \rfloor, \qquad M_3 \coloneqq \log^{\xi_3}(n), \qquad M_4 \coloneqq \lfloor \log^{\xi_1}(n) \rfloor. 
\end{equation*}
As before, we will choose $M_1$ so that $M_1 \gg M_2 \gg M_3 \gg M_4$.

We begin by defining the core inside a box $\Lambda_{M_4}(x), x \in M_4\mathbb{Z}^d$ as the largest connected component of the LRP truncated at length $N$. Observe that by Proposition~\ref{prop:Volume},  by choosing $\xi_1>0$ large enough, with probability larger than $1- \exp(-c\log(n)^2)$, these clusters have positive density in each $\Lambda_{M_4} \in \mathcal{T}_{M_4}(\Lambda_n)$. More precisely, with probability larger than $1- \exp(-c\log(n)^2)$, there exists $\rho_1>0$ such that $|\mathcal{G}(\Lambda_{M_4})| \ge \rho_1 M_4^d$ in every $\Lambda_{M_4} \in \mathcal{T}_{M_4}(\Lambda_n)$. We then reveal all edges between the vertices contained in $\{\mathcal{G}(\Lambda_{M_4})\}_{\Lambda_{M_4} \in \mathcal{T}_{M_4}(\Lambda_n)}$. 
We observe that with probability larger than $1- \exp(-c\log(n)^2)$ the second largest connected component is smaller than $M_4$ by \eqref{eqn:SecondLargTrunc}, and hence for $\xi_1>0$ large enough the sets $\mathcal{G}(\Lambda_{M_4})$ are also part of the largest connected component in $\Lambda_n$ (in the $N$-truncated LRP) and hence they must be connected. The core has the properties that
    \begin{enumerate}
        \item $|\mathcal{C}_1^{(n), N}(\Lambda_{M_4}) \cap \Lambda_{M_4}| \ge \rho_1 M_4^{d}$ for all $\Lambda_{M_4} \in \mathcal{T}_{M_4}$.
        \item $|E(\mathcal{C}_1^{(n), N}(\Lambda_{M_4}), \Lambda_n))| \le \rho_2 M_4^d$ for all $\Lambda_{M_4} \in \mathcal{T}_{M_4}$.
    \end{enumerate}
We define
\begin{equation*}
    \mathcal{G}_{4} \coloneqq \bigcup_{\Lambda_{M_4} \in \mathcal{T}_{M_4}} \mathcal{G}(\Lambda_{M_4}).
\end{equation*}
    
\subsubsection{Allocating small clusters to the core}

The allocation will be split into two steps. Firstly, we notice that for any vertex $v$ inside $\Lambda_{M_4}$ the probability that it connects by an edge to the core when revealing the edges (longer than $N$) of length up to $M_3=\log^{\xi_3}(n)$ with $\xi_3 > \xi_1$ is, by only looking at the edges of length less than $5\log^{\xi_1}(n)>N$, for some $c>0$ at least 
\begin{align*}
    1 - \exp \left( c \frac{\log^{d \xi_1}(n)}{\log^{\xi_1 s}(n)} \right).
\end{align*}
Hence, the probability to create a component which is larger than $\log^{\xi_4}(n)$ that does not connect to the core is, by independence, less than
\begin{equation}\label{eqn:DistancesLarge}
    \exp \left( - c \frac{\log^{d \xi_1 + \xi_4}(n)}{\log^{\xi_1 s}(n)} \right).
\end{equation}
Here, the crucial point is that we can choose $\xi_4>0$ large enough without the need to commit to a choice for $\xi_3>0$. We can now fix $\xi_4>0$ so that the second largest component is smaller than $\log^{\xi_4}(n)$ with probability larger than $1- \exp(-c\log^2(n))$. Let us now comment on the Euclidean distance between vertices $v$ that connect to the core via this new revealment of edges and the core $\mathcal{G}_{4}$. We showed that for any $v$, by exploring its connected component, we find that every $\log^{\xi_4}(n)$ number of steps outside $\mathcal{G}_4$ there must exist one edge that connects directly to the core. Furthermore, all edges are of length at most $\log^{\xi_3}(n)$ and hence the Euclidean distance between $v$ and the closest vertex in the core bounded above by $\log^{\xi_4+\xi_3}(n)$. 

We now fix $\xi_2 = \xi_4 + \xi_3 + 10d$. 

Similarly to the previous subsection we construct an allocation function $\phi_1$ that associates to a vertex $x$ a box in $\mathcal{T}_{M_2}$. We define the allocations at this stage in the exact same way that we did in Section~\ref{sect:SmallToCoresSmall}. As a result we obtain components $\{\mathcal{G}(3, \Lambda_{M_2})\}_{\Lambda_{M_2} \in \mathcal{T}_{M_2}}$.

We collect the properties of these sets in the following result. Recall the definitions of the $\sigma$-field given in \eqref{eqn:SigmaH}, \eqref{eqn:SigmaF} that can be naturally extended to the present context. Furthermore, for $j = 1, 2, 3, 4$ let $\mathcal{C}(j, \mathcal{G}_{4})$ be the connected component of $\mathcal{G}_{4}$ once all edges in $\mathcal{H}_{j}$ have been revealed.

\begin{prop}\label{prop:SecondAllocEucl}
    There exist constants $\rho_1, \rho_2 > 0$ and $\Delta_5, c_5>0$ such that for all $n$ large enough there exists an event $\mathcal{E}$ measurable w.r.t.\ $\mathcal{H}_{M_3}$ on which the partition $\{\mathcal{G}(3, \Lambda_{M_2})\}_{\Lambda_{M_2} \in \mathcal{T}_{M_2}}$ of $\mathcal{C}(3, \mathcal{G}_{4})$ into connected components has the property 
    \begin{equation*}
        \rho_1 M_2^d \le |E(\mathcal{G}(3, \Lambda_{M_2}), \Lambda_n)| \le \rho_2 M_2^d, \qquad \sup_{x \in \mathcal{C}(3, \mathcal{G}_{4})}d_{\mathcal{G}(3, \Lambda_{M_2})}(x, \mathcal{G}_4) \le \log^{\Delta_5}(n).
    \end{equation*}
    Moreover,
    \begin{equation*}
        \mathbf{P}\left( \mathcal{E} \mid \mathcal{F}_2 \right) \ge 1 - e^{-c_5 \log^2(n)}.
    \end{equation*}
\end{prop}
\begin{proof}
    The connectedness and distance bound follow directly from the construction and discussion around \eqref{eqn:DistancesLarge}. The volume lower bound follows from the volume lower bound of the largest component of the truncated percolation, see Proposition~\ref{prop:Volume}. We are left with the task of proving the volume upper bound. By the choice we made for $\xi_2$ we know from Lemma~\ref{lemma:VolLargeSets} and a union bound that each box has volume at most
    \begin{equation*}
        |E(\mathcal{G}(3, \Lambda_{M_2}), \Lambda_n)| \le \rho_2 M_2^d.
    \end{equation*}
    We observe that, on $\mathcal{E}$, a vertex $x$ is associated to a box which is at most at Euclidean distance $\log^{\xi_4 + \xi_3}(n)$ from its location. However, the side length of $M_2$ boxes is $\log^{\xi_2}(n)$ with $\xi_2 = \xi_4 + \xi_3 + 10d$. Hence, each box in the partition does not get assigned any vertex further than in the $4^d$ boxes that are $*$-adjacent to it. The total volume is then bounded readily by 
    \begin{equation*}
        |E(\mathcal{G}(3, \Lambda_{M_2}), \Lambda_n)| \le \rho_2 M_2^d + 4^d \rho_2 M_2^d.
    \end{equation*}
    This is enough to conclude by adjusting the value of $\rho_2>0$.
\end{proof}
In the final allocation phase we reveal all remaining edges, i.e.\ the ones that are longer than $M_3$ and that were not between vertices inside $\mathcal{C}(3, \mathcal{G}_{4})$. Let $\mathscr{C}_j, j \in \mathscr{S}$, where $\mathscr{S}$ denotes the finite index set of small clusters once we reveal all edges in $\mathcal{H}_{M_3}$. We immediately observe that
\begin{itemize}
    \item $|\mathscr{C}_j| \le \log^{\xi_4}(n)$ for all $j \in \mathscr{S}$.
    \item $|\mathscr{S}|\le n^d$.
\end{itemize}
We define an allocation function in the same way we did in the previous allocation step, using graph distance with a deterministic rule to break ties.

\begin{prop}
    There exist constants $\rho_1, \rho_2 > 0$ and $\Delta_6, c_6>0$ such that for all $n$ large enough there exists an event $\mathcal{B}$ on which the partition $\{\mathcal{G}(0, \Lambda_{M_2})\}_{\Lambda_{M_2} \in \mathcal{T}_{M_2}}$ of $\mathcal{C}(0, \mathcal{G}_4)$ into connected components has the property 
    \begin{equation*}
        \rho_1 M_2^d \le |E(\mathcal{G}(0, \Lambda_{M_2}),\Lambda_n)| \le \rho_2 M_2^d, \qquad \sup_{x \in \mathcal{C}(0, \mathcal{G}_{4})}d_{\mathcal{G}(0, \Lambda_{M_2})}(x, \mathcal{G}_4) \le \log^{\Delta_6}(n).
    \end{equation*}
    Moreover,
    \begin{equation*}
        \mathbf{P}\left( \mathcal{B} \mid \mathcal{E}, \mathcal{H}_{M_3} \right) \ge 1 - e^{-c_6 \log^2(n)}.
    \end{equation*}
\end{prop}
\begin{proof}
     As we already did in Proposition~\ref{prop:BranchingSmalls}, we will compare the number of $\mathscr{C}_j, j \in \mathscr{S}$ added to the already existing sets $\mathcal{G}(3, \Lambda_{M_2})$ to a subcritical branching process. We only describe the necessary adaptations.
     
     Observe that, while we have fixed $\xi_1, \xi_2>0$, we have not fixed yet a value for $\xi_3>0$; this will be done along the proof. The branching process associated to the connections between sets $\mathscr{C}_j, \mathscr{C}_i  \in \mathscr{S}$ has mean offspring
     \begin{equation*}
       \eta_n = \log^{\xi_4}(n) \log^{(d-s) \xi_3}(n).
    \end{equation*}
    We can then follow almost verbatim the steps already showed within the proof of Proposition~\ref{prop:BranchingSmalls} and ultimately obtain, with probability larger than $1 - \exp(-c\log^2(n))$, that the total number of vertices allocated to any $\mathcal{G}(3, \Lambda_{M_2})$ in this final allocation step is less than
    \begin{equation*}
        \log^4(n) \log^{\xi_4}(n) \log^{\xi_3(d-s)}(n)\log^{\xi_2 d}(n) \le \rho_2\log^{\xi_2 d}(n),
    \end{equation*}
    as $s>2d$. This concludes the proof.

\end{proof}

\section{The isoperimetric inequality and the upper bound on the mixing time}\label{sect:UB}

Theorem \ref{thm:LBIsoGen} will be established by a case distinction between sets, based on their geometric properties. 

\subsection{Large sets in the truncated giant}

We first deal with large sets in the truncated model. Recall that $\mathcal{C}_1^{(n)}$ is the largest component of the full long-range percolation model while $\mathcal{C}_1^{(n), N}$ is the one of the $N$-truncated percolation. We then have the following result:

\begin{lem}
There exist constants $c, c_{\mathrm{iso}}, \zeta>0$ such that for all $n$ large enough, with $\mathbf{P}$-probability larger than $1 - \exp(-c\log^2(n))$, the following property holds: for all sets $A \subseteq \mathcal{C}_1^{(n)}$ with the property that 
%\dmc{Am I missing something? I don't think we can assume this. If the other two cases below do not happen, we only know that at least say $1/2$ of $|E(A,\Lambda_n)|$ is contained in the giant of the truncated model. But I suppose you can repair the proof with looking only at $A'\subseteq A$ that is already contained in the truncated model, which gives these edges? Please change it if you agree} 
$c_{b} \log(n)^{\zeta} \le |E(A, \Lambda_n)| \le |E(\mathcal{C}_1^{(n)}, \Lambda_n)|/2$ and that at least $1/2$ of the edges of the set $E(A,\Lambda_n)$ is contained in $E(\mathcal{C}_1^{(n), N}, \Lambda_n)$, we have 
    \begin{equation*}
        |E(A, A^c)| \ge c_{\mathrm{iso}} |E(A, \Lambda_n)|^{\frac{d-1}{d}}.
    \end{equation*}
\end{lem}
\begin{proof}
%\dmc{Why do we need this first case? The second part of this lemma contains it anyway, no? Indeed the current Lemma 6.3 is very similar, and these cases can be merged, with the case analysis that you already have here in the second part it contains everything}    
     Consider $A\subseteq \mathcal{C}_1^{(n), N}$ with the properties stated in the lemma. Consider the tessellation of $\Lambda_n$ using boxes of side length $M=\log^{\zeta_1/d}(n)$ for $\zeta_1>0$. Note that we always have $|E(A, A^c)| \ge |E_N(A, A^c)|$, since adding long edges cannot make the number of edges from $A$ to $A^c$ smaller. However, we might also have $|E(A, \Lambda_n)| \gg |E_N(A, \Lambda_n)|$; we show that this is not the case. 
     
     Choose $\zeta_1>0$ so that, by Proposition \ref{prop:Volume},
     \begin{equation*}
         \mathbf{P}\left( \exists \Lambda_M \in \mathcal{T}_M \colon |\mathcal{C}_1^{(n), N} \cap \Lambda_M| \le \rho_1 M^d \right) \le Ce^{-cM^{d-1}} \le Ce^{-c\log^2(n)}.
     \end{equation*}
     Moreover, by Lemma \ref{lemma:VolLargeSets} there exists $\rho_2>0$ such that, by a union bound,
    \begin{equation}\label{eqn:NoLargeVol}
        \mathbf{P}\left(\exists \Lambda_M \in \mathcal{T}_M \subseteq \mathcal{C}_1^{(n), N} \colon |E(\Lambda_M,\Lambda_n)|\ge \rho_2 M^d \right) \le Ce^{-c\log^2(n)}.
    \end{equation}
    Let $\Lambda_M^{(i)}, i = 1, \dots, r$ be the boxes containing vertices of $A$. On the event inside~\eqref{eqn:NoLargeVol}, we have $r \ge |E_N(A, \Lambda_n)| \log^{-\zeta_1}(n)$. 
    
    First, consider the case in which $A$ occupies less than $1/2$ of the mass in more than $1/2$ of the $\Lambda_M^{(i)}, i = 1, \dots, r$ it is contained in. Then $A$ has at least one boundary edge in each of those boxes. Necessarily, the total boundary is 
    \begin{equation*}
        |E_N(A, A^c)| \ge c_{\mathrm{iso}} |E_N(A, \Lambda_n)|\log^{-\zeta_1}(n) \ge c_{\mathrm{iso}}|E(A,\Lambda_n)|^{(d-1)/d},
    \end{equation*}
    for $\zeta_1 < \zeta/d$.
    
    Otherwise, consider the case in which $A$ occupies at least $1/2$ of the mass in more than $1/2$ of the $\Lambda_M^{(i)}, i = 1, \dots, r$ it is contained in. Then, we know that $A$ accumulates volume $\rho_1 M^d/2$ in each such box. Then, since in all boxes the edge volume is comparable to the vertex volume (see \eqref{eqn:NoLargeVol}), a positive proportion of the mass of $A$ is necessarily in these sets. 
    
    By Proposition \ref{prop:IsoTruncated} we obtain that
    \begin{equation*}
        |E(A,A^c)| \ge |E_N(A,A^c)| \ge c_{\mathrm{iso}} |A|^{(d-1)/d} \ge 
    c_{\mathrm{iso}} |E(A,\Lambda_n)|^{(d-1)/d}.
    \end{equation*}
    The proof is concluded.
    % First we consider the case in which $A$ is connected within $\mathcal{C}_1^{(n), N}$. By Proposition~\ref{prop:IsoTruncated} we have that (within the graph $\mathcal{C}_1^{(n), N}$) \dmc{I would put in the Proposition 4.1 only $|A|^{(d-1)/d}$, as I mentioned there, and here if you need it you can put $|E_N(A, \Lambda_n)|^{\frac{d-1}{d}}$}
    % \begin{equation*}
    %     |E_N(A, A^c)| \ge c_{\mathrm{iso}} |E_N(A, \Lambda_n)|^{\frac{d-1}{d}} \ge c_{\mathrm{iso}} |A|^{\frac{d-1}{d}}.
    % \end{equation*}

    % To deal with this, we first observe that \dm{there exists $c=c(N,d) > 0$}   such that the number of lattice animals (connected sets) of size $k$ in the $N$-truncated percolation models is bounded from above by $c(N, d)^k$. Then, by Lemma~\ref{lemma:VolLargeSets} we can find $\mu^*>0$ sufficiently large such that 
    % \begin{equation*}
    %     \mathbf{P}\left(\sum_{x \in A} \mathrm{deg}^\omega(x) \ge \mu^* |A|\right) \le e^{- (c(N, d) + 1) |A|}.
    % \end{equation*}
    % Thus, 
    % Thus, by the Borel-Cantelli lemma, with $\mathbf{P}$-probability larger than $(1-\varepsilon)$, for any set $A$ with $|A| \ge c_{b}\log(n)^{\frac{d}{d-1}}$, we have $|E(A,\Lambda_n)| \le \mu |A|$. Hence,
    % $$ 
    % |E(A,A^c)| \ge |E_N(A,A^c)| \ge c_{\mathrm{iso}} |A|^{(d-1)/d} \ge 
    % c_1 |E(A,\Lambda_n)|^{(d-1)/d},
    % $$
    % and the result follows.

\end{proof}

\subsection{Small sets outside the giant}
It remains to deal with sets which contain at least a constant fraction of its incident edges in components of the $N$-truncated percolation not connected to the giant truncated component. We will here perform the most delicate step in the proof, towards which most of our preparations were directed.
\begin{figure}[H]
    \centering
    \includegraphics[width=0.55\linewidth]{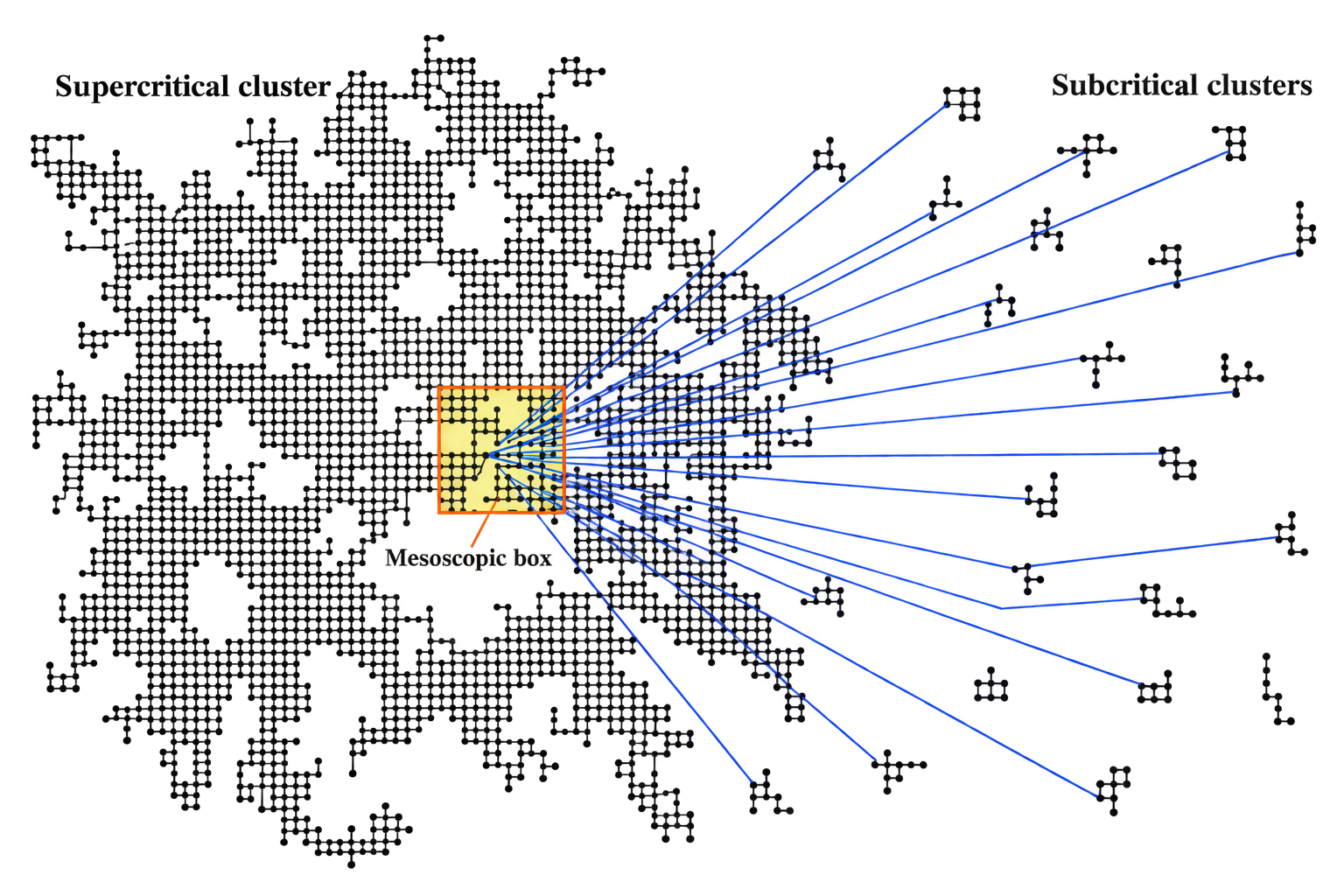}
    \caption{The following situation is the one we need to avoid in this proof and that requires the allocation theorem: many small clusters get assigned to the same mesoscopic box. Observe that the blue edges cannot alter the volume of the box, but the attached microscopic clusters might (in principle, by a multiplicative term of polylogarithmic size).}
    \label{fig:FigFar}
\end{figure}

The main result of the section will be the following lemma.

\begin{lem}
    For any $\zeta>0$, there exists a constant $c, c_{\mathrm{iso}}>0$ such that, with $\mathbf{P}$-probability larger than $1 - \exp(-c\log^2(n))$ the following happens: for all sets $A \subseteq \mathcal{C}_1^{(n)}$ with $n^{\zeta} \le |E(A, \Lambda_n)| \le |E(\mathcal{C}_1^{(n)}, \Lambda_n)|/2$ such that at least $1/4$ of $|E(A, \Lambda_n)|$ is contained in connected components $\{U_i\}_{i \ge 0}$ of the $N$-truncated percolation of size $|U_i|\le c_{b} \log(n)^{\Delta_1}$ not connected to $\mathcal{C}_1^{(n), N}$, we have
    \begin{equation*}
        |E(A, A^c)| \ge c_{\mathrm{iso}} |E(A, \Lambda_n)|^{\frac{d-1}{d}}.
    \end{equation*}
\end{lem}

\begin{proof}
    Firstly, we observe that by Lemma~\ref{lemma:SecondLargest} the connected components in the truncated percolation models that are not in the giant are almost surely at most of size $\log^{\Delta_1}(n)$ for some $\Delta_1>0$. Indeed, for a well chosen $\Delta_1>0$, by Lemma~\ref{lemma:SecondLargest} one obtains that $|\mathcal{C}_2^{(n), N}| \le c_b \log^{\Delta_1}(n)$ $\mathbf{P}$-a.s.\ for all $n$ large.
    
    Now, let $A_1, \dots, A_{r}$ with $r \ge |E(A, \Lambda_n)|^{1 - \delta}$, for arbitrarily small $\delta>0$, be the components in the $N$-truncated percolation outside the giant of the truncated percolation, including singletons. The bound on $r$ follows as before, since each component has at most $c_{b} \log(n)^{\Delta_1}$ sites and each site has at most $\mu \log(n)$ many edges incident to itself. 

    By Theorem~\ref{theo:Partition}, for each $x \in \mathcal{C}_1^{(n)}$, for all $M \ge n^\eta$ (the value of $\eta>0$ will be fixed later to be small enough to ensure certain properties) there exists some $\Lambda_M$ such that $d_{\mathcal{C}_1^{(n)}}(x, \mathcal{C}_1^{(n), N}(\Lambda_M)) \le \log^{\Delta_1}(n)$ and $x \in \tilde{\mathcal{C}}^{(n)}(\Lambda_M)$ (recall that $\tilde{\mathcal{C}}^{(n)}(\Lambda_M)$ denotes the connected set of vertices allocated to the box $\Lambda_M \in \mathcal{T}_M$).

    We now distinguish edges according to certain properties:
    \begin{enumerate}
        \item We call the edges from $A_i$ to $\mathcal{C}_1^{(n), N}$ that land on a vertex inside $A \cap \mathcal{C}_1^{(n), N}$ {\color{blue}blue}.
        \item For each $i= 1, \dots, r$, from each vertex $x \in A_i$, by Theorem~\ref{theo:Partition}, there exists an open path of length at most $\log^{\Delta_1}(n)$ that connects $x$ to $\mathcal{C}_1^{(n), N}$. Furthermore, such path stays inside the random sets $\tilde{C}^{(n)}(\Lambda_M)$ introduced in Theorem~\ref{theo:Partition}. Let us also recall that the existence of this path defines the allocation function $\phi$ in Theorem~\ref{theo:Partition}. We call the edges that connect the sets $A_i$ to  $\mathcal{C}_1^{(n), N}$ {\color{red}red}.
    \end{enumerate}
    Observe that each edge from $\cup_{i =1, \dots, r} A_i$ to $\mathcal{C}_1^{(n), N}$ can be only {\color{red}red}, only {\color{blue}blue}, not colored or both {\color{red}red} and {\color{blue}blue}. 

    \begin{figure}[H]
    \centering
    \includegraphics[width=0.55\linewidth]{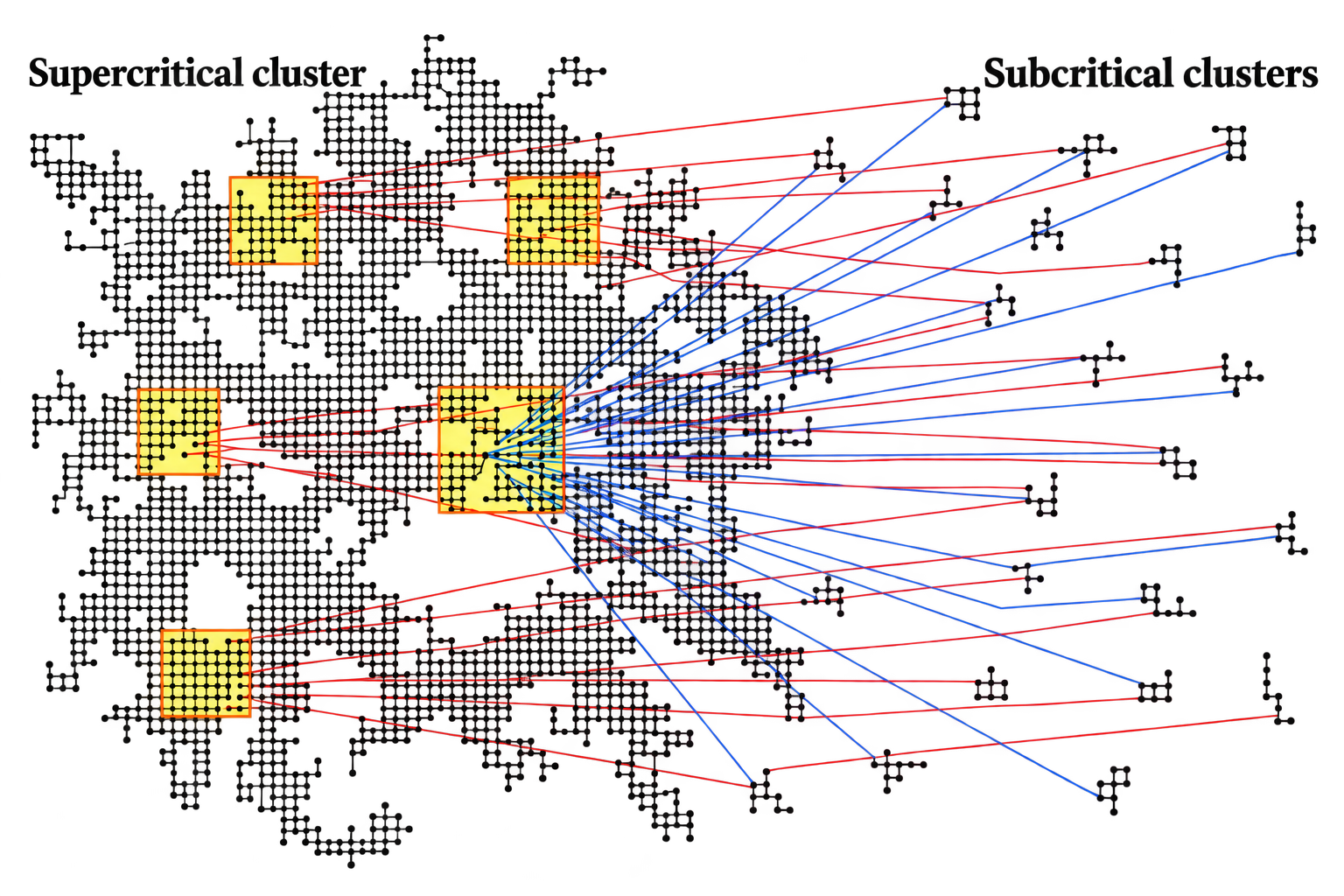}
    \caption{How the allocation theorem solves our issues: even if the situation of the blue edges happens, there are many red edges distributing the mass of the small vertices. Then these red edges either add boundary edges or they distribute the mass of microscopic clusters.}
    \label{fig:FigFar}
\end{figure}
    
    By Theorem~\ref{theo:Partition}, each $x \in A_i$, is within graph distance $\log^{\Delta_1}(n)$ from a {\color{red}red} edge and hence from a vertex $y \in \mathcal{C}_1^{(n), N}$. This implies readily that there exists some $c, \Delta_3>0$ such that the total number of {\color{red}red} edges is at least $c |E(A, \Lambda_n)| \log^{-\Delta_3}(n)$. 

    Moreover, let $B_i, i= 1, \dots, \tilde{r}$ be the components of $A$ inside distinct boxes of side length $M= n^{\eta}$ that receive {\color{red}red} edges. Observe that each of these components is necessarily attached to $\mathcal{C}^{(n), N}_1$, but they could also connected among themselves by longer edges. Formally, $B_i \subseteq A \cap \mathcal{C}_1^{(n), N}(\Lambda_M^{(i)})$ with $\Lambda_M^{(i)} \in \mathcal{T}_M$, 
    %\dmc{I don't understand the next sentence, how can there be two edges according to the rule of the allocation? I'd leave it out} \cs{in the box, there could be multiple red edges coming from points in $A$. The allocation rule ensures the existence of at least one.} 
    and in the case that there are two (or more) {\color{red}red} edges from $A_j, j=1, \dots, r$ to two (or more) distinct (i.e.\ not directly connected by an edge) $B_i^{(1)}, B_i^{(2)} \subseteq A \cap \mathcal{C}_1^{(n), N}(\Lambda_M^{(i)})$, then choose one to be $B_i$ with some deterministic rule (e.g.\ containing the minimal vertex in the lexicographic ordering).  
    We highlight that in the case in which all {\color{red}red} edges from $A_j, j=1, \dots, r$ landing in $\mathcal{C}_1^{(n),N}(\Lambda_M^{(i)})$ are not {\color{blue}blue}, we consider $B_i$ to be equal to the empty set (that means that these {\color{red}red} edges land in the set $\mathcal{C}_1^{(n), N}\setminus A$).
    
 Theorem~\ref{theo:Partition} implies that the set of allocated vertices (and relative edges) going to each box is at most $(\rho_2 - \rho_1)M^d$. This implies that there at least $\tilde{r} \ge c|E(A, \Lambda_n)| n^{-d\eta}$ components $B_i$. Moreover, since $|\mathcal{C}_1^{(n), N}(\Lambda_M^{(i)})| \ge \rho_1M^d$ in every $\Lambda_M^{(i)} \in \mathcal{T}_M$, we obtain 
    \begin{equation}\label{eqn:Mass}
        \bigg| E_N\bigg(\bigcup_{i = 1}^{\tilde{r}} \mathcal{C}_1^{(n), N}(\Lambda_M^{(i)}), \Lambda_n \bigg)\bigg| \ge \frac{\rho_1}{\rho_2}|E(\cup_{i = 1}^{r}A_i, \Lambda_n)| \ge \frac{1}{2} \frac{\rho_1}{\rho_2}|E(A, \Lambda_n)|.
    \end{equation}
    Assume now that for a proportion of more than $1/2$ of the $B_i, i= 1, \dots, \tilde{r}$, we have $B_i = \emptyset$. This necessarily implies that for each empty box there exists a {\color{red}red} edge which is not {\color{blue}blue}. Observe that, by definition, all such edges are boundary edges of $A$. In this case, the total boundary of $A$ is at least
    \begin{equation*}
        |E(A, A^c)| \ge c |E(A, \Lambda_n)| n^{-\eta d},
    \end{equation*}
 which exceeds $c|E(A, \Lambda_n)|^{\frac{d-1}{d}}$ for all $|A| \ge n^{\zeta}$, if $\eta>0$ is chosen small enough w.r.t. $\zeta>0$.
 
 Now assume that for a proportion of more than $1/2$ of the $B_i, i= 1, \dots, \tilde{r}$, we have $B_i \neq \emptyset$. We distinguish two further cases:
     \begin{itemize}
        \item[(A.1)] The $B_i \neq \emptyset$ are such that 
        \begin{equation*}
            |E_N(B_i, \Lambda_n)| \ge \frac{1}{2} |E_N( \mathcal{C}_1^{(n), N}(\Lambda_M^{(i)}), \Lambda_n)|,
        \end{equation*}
        for at least $\tilde{r}/4$ boxes.
        \item[(A.2)] The $B_i \neq \emptyset$ are such that 
        \begin{equation*}
            |E_N(B_i, \Lambda_n)| < \frac{1}{2} |E_N( \mathcal{C}_1^{(n), N}(\Lambda_M^{(i)}), \Lambda_n)|,
        \end{equation*}
        for at least $\tilde{r}/4$ boxes.
    \end{itemize}
    In case (A.1), the component $B = \cup_i B_i$ has total edge mass comparable (up to a multiplicative constant $1/4$) with $\cup_i\mathcal{C}_1^{(n), N}(\Lambda_M^{(i)})$ and hence, by \eqref{eqn:Mass} we also have 
    \begin{equation}
        |E_N(B, \Lambda_n)| \ge  \frac{\rho_1}{8\rho_2}|E(A, \Lambda_n)|.
    \end{equation}
    Proposition~\ref{prop:IsoTruncated} implies that its boundary is larger than
    \begin{equation*}
        |E_N(B, B^c)| \ge c |E_N(B, \Lambda_n)|^{\frac{d-1}{d}},
    \end{equation*}
    and since $|E_N(B, B^c)| \le |E(A, A^c)|$ we obtain that
    \begin{equation*}
        |E(A, A^c)| \ge c |E(A, \Lambda_n)|^{\frac{d-1}{d}},
    \end{equation*}
    as required. Finally, if (A.2) happens then in each box in which $B_i$ takes less than $1/2$ the volume there is at least one boundary edge for $A$, hence
    \begin{equation*}
        |E(A, A^c)| \ge c |E(A, \Lambda_n)| n^{-\eta d},
    \end{equation*}
    which concludes our case distinction by choosing $\eta < \zeta/d^2$.

\end{proof}

\begin{proof}[Proof of Theorem \ref{theo:MixingDiffusive}]
    The following statements hold with high $\mathbf{P}$-probability. In particular, we consider the event of Theorem \ref{thm:LBIsoGen} to hold throughout the proof. Observe that by the Borel-Cantelli lemma, such event holds almost surely for all $n$ large enough. We argue along the lines of \cite{BenjaminiMossel}.

    We begin by observing that with high probability $\rho_1 n^d \le |E(\mathcal{C}_1^{(n)}, \Lambda_n)| \le \rho_2n^d$.
    
    Furthermore, Theorem \ref{thm:LBIsoGen} implies that for all $A \subseteq \mathcal{C}_1^{(n)}$ such that  $n^{\eta-d} \le \pi(A) \le 1/2$ we have that
    \begin{equation*}
        \varphi(A) \ge   \frac{Q(A, A^c)}{\pi(A) \pi(A^c)} \ge \frac{|E(A, A^c)|}{|E(\Lambda_n, \Lambda_n)| \pi(A)} \ge \frac{|E(A, \Lambda_n)|^{1 - \frac{1}{d}}}{|E(\Lambda_n, \Lambda_n)| \pi(A)} \ge \frac{(\rho_1\pi(A) n^d)^{1-\frac{1}{d}}}{\rho_2n^d \pi(A)} \ge \frac{c}{n\pi(A)^{1/d}}.
    \end{equation*}
    By the fact that $\mathcal{C}_1^{(n)}$ is connected, for $A$ with $\pi(A) \le n^{\eta-d}$ we obtain
    \begin{equation*}
        \varphi(A) \ge \frac{c}{n^{\eta}}.
    \end{equation*}
    We now claim that, for all $t \in (0, 1/2)$,
    \begin{equation}\label{eqn:ClaimIso}
        \varphi(t) \ge \min\left\{\frac{c}{n^{\eta}}, \frac{c}{n t^{1/d}}\right\}.
    \end{equation}
    First, observe that \eqref{eqn:ClaimIso} holds for $t=1/2$ as the minimum is achieved with both $A, A^c$ connected. Moreover, by adjusting the constant $c>0$ we can (by monotonicity) extend the claim to $q \le t \le 1/2$ for any fixed $0<q<1/4$, whose precise value will be fixed later on.
    
    Fix $0<t \le q$, and let $A$ be the connected set with minimal $\varphi(A)$ among those with $\pi(A) \le t$. If $A$ is such that $ \pi(A) \le n^{\eta - d}$ then \eqref{eqn:ClaimIso} follows directly, thus we assume this does not hold. Furthermore, we assume that $\varphi(A) \le c/(nt^{1/d})$, otherwise the required inequality also holds. Let $B = A^c$ and let $B_i$, $i=1, \dots, r$ be the connected components (disconnected from each other by removing $A$) of $B$ ordered so that $\pi(B_1) \ge \pi(B_2) \ge \dots$. Assume first $\pi(B_1) \le \pi(A)$. In this case there exists a union of $B_i$ such that $1/4 \le \pi(\cup_i B_i) \le 1/2$. We obtain 
    \begin{equation*}
        \frac{c}{n^{1/d}} \le \varphi(1/2) \le \varphi(\cup_i B_i) \le 8 Q(A, A^c)  \le 8c\varphi(A)\pi(A) \le 8 \frac{ct^{1-1/d}}{n} \le \frac{4cq^{1 - 1/d}}{n},
    \end{equation*}
    which leads to a contradiction for $q$ small enough. Now assume $\pi(B_1) \ge \pi(A)$. Let $B'$ be the smaller set among $B_1$ and $B_1^c$, i.e.\ $B' = \argmin \{\pi(B_1), \pi(B_1^c)\}$. Observe that, by construction, they are both connected and that $\pi(A) \le \pi(B') \le 1/2$. This implies that $\varphi(B') \ge c/(\pi(B')^{1/d}n)$, which yields that
    \begin{equation*}
        Q(A, A^c) \ge Q(B', B'^c) \ge \frac{c\pi(B')^{1-1/d} (1 - \pi(B'))}{n} \ge \frac{c\pi(A)^{1-1/d} (1 - \pi(A))}{n}.
    \end{equation*}
    This directly implies that $\varphi(A) \ge c/(nt^{1/d})$, concluding the proof of \eqref{eqn:ClaimIso}. Plugging \eqref{eqn:ClaimIso} into \eqref{eqn:CondMeth} yields the required bound as we can choose $\eta>0$ such that
    \begin{equation*}
        n^{2\eta}\int_{1/n}^{1/2} \frac{1}{t} dt \le Cn^{2\eta} \log(n) \le C n^2.
    \end{equation*}
\end{proof}

\section{The model with nearest neighbour edges and $d=1$}\label{sect:D1}

In dimension $d=1$, it is a classical result that for $s > 2$ one needs to assume that all edges of length $1$ are present in order to percolate. We will thus naturally assume that $\mathbf{p}(x, y) = 1$ if $y \in \{x-1, x+1\}$. The lower bound for the mixing time is proved in \cite{BergerBenjaminiYadin}. The upper bound follows from the following basic lemma and \cite[Theorem 2]{Oliveira}.
\begin{lem}
    Let $d=1$ and $s>2$. There exists a constant $C>0$ such that $\mathbf{P}$-a.s.\ for all $n$ large enough,
    \begin{equation*}
        \sup_{x \in \Lambda_n, A \colon \pi(A)\ge 1/3} E^\omega_0[T_A] \le Cn^2.
    \end{equation*}
\end{lem}
\begin{proof}
By the ergodic theorem for stationary sequences, $\mathbf{P}$-a.s.\ for all $n$ large enough, $|E(\Lambda_n, \Lambda_n)| \le \mu n$ for $\mu > \mathbf{E}[\mathrm{deg}^\omega(0)]$. On this event, clearly, there does not exist any $A$ with more than $\mu n$ total edges. Observe that the supremum of the hitting time over all sets with $\pi(A) \ge 1/3$ coincides with the supremum of the exit time over sets $B$ with $\pi(B) \le 2/3$. For any such set $B$ and any $x \in \Lambda_n$, by the commute time identity (see \cite[Proposition~10.7]{PeresLevin}) we obtain 
\begin{equation*}
    E_x^\omega[T_{B^c}] \le 2 |E(B, B)| R(x, B^c),
\end{equation*}
where $R(x, B^c)$ denotes the effective resistance (see \cite[Section 9.4]{PeresLevin} for the formal definition). Deterministically, $R(x, B^c)$ is bounded uniformly over all $x$ and $B$ by the graph distance, which is at most $n$. It follows immediately that $\mathbf{P}$-a.s.\ for all $n$ large enough
\begin{equation*}
        \sup_{x \in \Lambda_n, A \colon \pi(A)\ge 1/3} E^\omega_0[T_A] \le 2\mu n^2.
    \end{equation*}
    The proof is concluded.
\end{proof}

\noindent \textbf{Acknowledgments}: This work was facilitated by DM's research visit to TU Munich. The authors gratefully acknowledge PUC Santiago and TU Munich for making this visit possible. Dieter Mitsche was supported by ANID-FONDECYT regular grant 1261336. \\

\noindent \textbf{AI Disclosure}: ChatGPT Astra was used to generate the two figures. LLMs were not used in any other capacity; in particular, the arguments and their presentations were derived and written entirely by the authors.
\bibliographystyle{plain}
\bibliography{ref}

@article{liggettdomination,
author={Liggett, Thomas M. and Schonmann, Roberto H. and Stacey, Alan M.},
year={1997}, 
title={Domination by product measures}, 
journal={The Annals of Probability}, volume={25},
number={1},
pages={71--95}
}

@book {PeresLevin,
    AUTHOR = {Levin, David A. and Peres, Yuval and Wilmer, Elizabeth L.},
     TITLE = {Markov chains and mixing times},
      NOTE = {With a chapter by James G. Propp and David B. Wilson},
 PUBLISHER = {American Mathematical Society, Providence, RI},
      YEAR = {2009},
     PAGES = {xviii+371},
      ISBN = {978-0-8218-4739-8},
   MRCLASS = {60J10 (60-01 60J05 60K35 60K37 68U20 68W20)},
  MRNUMBER = {2466937},
MRREVIEWER = {Olle\ H\"aggstr\"om},
       DOI = {10.1090/mbk/058},
       URL = {https://doi-org.tum-eaccess.de/10.1090/mbk/058},
}

@article {CrawfordSly1,
    AUTHOR = {Crawford, Nicholas and Sly, Allan},
     TITLE = {Simple random walk on long range percolation clusters {I}:
              heat kernel bounds},
   JOURNAL = {Probab. Theory Related Fields},
  FJOURNAL = {Probability Theory and Related Fields},
    VOLUME = {154},
      YEAR = {2012},
    NUMBER = {3-4},
     PAGES = {753--786},
      ISSN = {0178-8051,1432-2064},
   MRCLASS = {60K35 (05C81 60Fxx 60K37)},
  MRNUMBER = {3000561},
MRREVIEWER = {Andrew\ R.\ Wade},
       DOI = {10.1007/s00440-011-0383-2},
       URL = {https://doi-org.tum-eaccess.de/10.1007/s00440-011-0383-2},
}

@article {CrawfordSly2,
    AUTHOR = {Crawford, Nicholas and Sly, Allan},
     TITLE = {Simple random walk on long-range percolation clusters {II}:
              scaling limits},
   JOURNAL = {Ann. Probab.},
  FJOURNAL = {The Annals of Probability},
    VOLUME = {41},
      YEAR = {2013},
    NUMBER = {2},
     PAGES = {445--502},
      ISSN = {0091-1798,2168-894X},
   MRCLASS = {60G50 (60F17 60G52 60K35 60K37)},
  MRNUMBER = {3077517},
MRREVIEWER = {Andrew\ R.\ Wade},
       DOI = {10.1214/12-AOP774},
       URL = {https://doi-org.tum-eaccess.de/10.1214/12-AOP774},
}

@article {BenjaminiMossel,
    AUTHOR = {Benjamini, Itai and Mossel, Elchanan},
     TITLE = {On the mixing time of a simple random walk on the super
              critical percolation cluster},
   JOURNAL = {Probab. Theory Related Fields},
  FJOURNAL = {Probability Theory and Related Fields},
    VOLUME = {125},
      YEAR = {2003},
    NUMBER = {3},
     PAGES = {408--420},
      ISSN = {0178-8051,1432-2064},
   MRCLASS = {60G50 (60K35)},
  MRNUMBER = {1967022},
MRREVIEWER = {Sven\ Erick\ Alm},
       DOI = {10.1007/s00440-002-0246-y},
       URL = {https://doi-org.tum-eaccess.de/10.1007/s00440-002-0246-y},
}

@article {BaumlerShapeTheo,
    AUTHOR = {B\"aumler, Johannes},
     TITLE = {Continuity of the critical value and a shape theorem for
              long-range percolation},
   JOURNAL = {Stochastic Process. Appl.},
  FJOURNAL = {Stochastic Processes and their Applications},
    VOLUME = {199},
      YEAR = {2026},
     PAGES = {Paper No. 104982, 29},
      ISSN = {0304-4149,1879-209X},
   MRCLASS = {82B43 (05C12 05C81 60K35)},
  MRNUMBER = {5068923},
       DOI = {10.1016/j.spa.2026.104982},
       URL = {https://doi-org.tum-eaccess.de/10.1016/j.spa.2026.104982},
}

@article {BaumlerDistance1,
    AUTHOR = {B\"aumler, Johannes},
     TITLE = {Distances in {$\frac{1}{\Vert x-y\Vert^{2d}}$} percolation
              models for all dimensions},
   JOURNAL = {Comm. Math. Phys.},
  FJOURNAL = {Communications in Mathematical Physics},
    VOLUME = {404},
      YEAR = {2023},
    NUMBER = {3},
     PAGES = {1495--1570},
      ISSN = {0010-3616,1432-0916},
   MRCLASS = {60K35 (82B27 82B43)},
  MRNUMBER = {4677584},
       DOI = {10.1007/s00220-023-04861-z},
       URL = {https://doi-org.tum-eaccess.de/10.1007/s00220-023-04861-z},
}

@article {BergerLRP,
    AUTHOR = {Berger, Noam},
     TITLE = {Transience, recurrence and critical behavior for long-range
              percolation},
   JOURNAL = {Comm. Math. Phys.},
  FJOURNAL = {Communications in Mathematical Physics},
    VOLUME = {226},
      YEAR = {2002},
    NUMBER = {3},
     PAGES = {531--558},
      ISSN = {0010-3616,1432-0916},
   MRCLASS = {82B43 (60J45 60K35)},
  MRNUMBER = {1896880},
MRREVIEWER = {Olle\ H\"aggstr\"om},
       DOI = {10.1007/s002200200617},
       URL = {https://doi-org.tum-eaccess.de/10.1007/s002200200617},
}

@article{MitscheMartinezKiwi,
      title={Mixing time and isoperimetry in random geometric graphs}, 
      author={Marcos Kiwi and Carlos Martinez and Dieter Mitsche},
      year={2025},
      eprint={2510.19951},
      archivePrefix={arXiv},
      primaryClass={math.PR},
      journal={Preprint available at https://arxiv.org/abs/2510.19951}, 
}

@article{HaalandSarkovic,
      title={Mixing time of the random walk on the giant component of the random geometric graph}, 
      author={Magnus H. Haaland and An{\dj}ela \v{S}arkovi\'{c}},
      year={2025},
      eprint={2512.21322},
      archivePrefix={arXiv},
      primaryClass={math.PR},
      journal={Preprint available at https://arxiv.org/abs/2512.21322}, 
}

@article {BiskupDistance1,
    AUTHOR = {Biskup, Marek},
     TITLE = {On the scaling of the chemical distance in long-range
              percolation models},
   JOURNAL = {Ann. Probab.},
  FJOURNAL = {The Annals of Probability},
    VOLUME = {32},
      YEAR = {2004},
    NUMBER = {4},
     PAGES = {2938--2977},
      ISSN = {0091-1798,2168-894X},
   MRCLASS = {60K35 (05C80 82B28 82B43)},
  MRNUMBER = {2094435},
MRREVIEWER = {Olivier\ Raimond},
       DOI = {10.1214/009117904000000577},
       URL = {https://doi-org.tum-eaccess.de/10.1214/009117904000000577},
}

@article {BiskupDiameter,
    AUTHOR = {Biskup, Marek},
     TITLE = {Graph diameter in long-range percolation},
   JOURNAL = {Random Structures Algorithms},
  FJOURNAL = {Random Structures \& Algorithms},
    VOLUME = {39},
      YEAR = {2011},
    NUMBER = {2},
     PAGES = {210--227},
      ISSN = {1042-9832,1098-2418},
   MRCLASS = {60K35 (05C12 05C80)},
  MRNUMBER = {2850269},
MRREVIEWER = {Maria\ Deijfen},
       DOI = {10.1002/rsa.20349},
       URL = {https://doi-org.tum-eaccess.de/10.1002/rsa.20349},
}

@article {BiskupOscillation,
    AUTHOR = {Biskup, Marek and Krieger, Andrew},
     TITLE = {Arithmetic oscillations of the chemical distance in long-range
              percolation on {$\Bbb Z^d$}},
   JOURNAL = {Ann. Appl. Probab.},
  FJOURNAL = {The Annals of Applied Probability},
    VOLUME = {34},
      YEAR = {2024},
    NUMBER = {3},
     PAGES = {2986--3017},
      ISSN = {1050-5164,2168-8737},
   MRCLASS = {60K35 (05C80 82B43)},
  MRNUMBER = {4757496},
MRREVIEWER = {John\ Haslegrave},
       DOI = {10.1214/23-aap2031},
       URL = {https://doi-org.tum-eaccess.de/10.1214/23-aap2031},
}

@article {BiskupKumagaiQuenched,
    AUTHOR = {Biskup, Marek and Chen, Xin and Kumagai, Takashi and Wang,
              Jian},
     TITLE = {Quenched invariance principle for a class of random
              conductance models with long-range jumps},
   JOURNAL = {Probab. Theory Related Fields},
  FJOURNAL = {Probability Theory and Related Fields},
    VOLUME = {180},
      YEAR = {2021},
    NUMBER = {3-4},
     PAGES = {847--889},
      ISSN = {0178-8051,1432-2064},
   MRCLASS = {60F17 (60J45)},
  MRNUMBER = {4288333},
MRREVIEWER = {Roman\ Urban},
       DOI = {10.1007/s00440-021-01059-z},
       URL = {https://doi-org.tum-eaccess.de/10.1007/s00440-021-01059-z},
}

@article {KumagaiLocalLimit,
    AUTHOR = {Chen, Xin and Kumagai, Takashi and Wang, Jian},
     TITLE = {Quenched local limit theorem for random conductance models
              with long-range jumps},
   JOURNAL = {Ann. Probab.},
  FJOURNAL = {The Annals of Probability},
    VOLUME = {54},
      YEAR = {2026},
    NUMBER = {3},
     PAGES = {1478--1529},
      ISSN = {0091-1798,2168-894X},
   MRCLASS = {60K37 (60F17 60J45)},
  MRNUMBER = {5061899},
       DOI = {10.1214/25-aop1794},
       URL = {https://doi-org.tum-eaccess.de/10.1214/25-aop1794},
}

@article {AizenmanNewman,
    AUTHOR = {Aizenman, Michael and Newman, Charles M.},
     TITLE = {Discontinuity of the percolation density in one-dimensional
              {$1/|x-~y|^2$} percolation models},
   JOURNAL = {Comm. Math. Phys.},
  FJOURNAL = {Communications in Mathematical Physics},
    VOLUME = {107},
      YEAR = {1986},
    NUMBER = {4},
     PAGES = {611--647},
      ISSN = {0010-3616,1432-0916},
   MRCLASS = {82A43},
  MRNUMBER = {868738},
MRREVIEWER = {R.\ T.\ Smythe},
       URL = {http://projecteuclid.org.tum-eaccess.de/euclid.cmp/1104116233},
}

@incollection {ConductanceMethod,
    AUTHOR = {Lov\'asz, L\'aszl\'o{} and Kannan, Ravi},
     TITLE = {Faster mixing via average conductance},
 BOOKTITLE = {Annual {ACM} {S}ymposium on {T}heory of {C}omputing
              ({A}tlanta, {GA}, 1999)},
     PAGES = {282--287},
 PUBLISHER = {ACM, New York},
      YEAR = {1999},
      ISBN = {1-58113-067-8},
   MRCLASS = {68W99 (65C50)},
  MRNUMBER = {1798047},
       DOI = {10.1145/301250.301317},
       URL = {https://doi-org.tum-eaccess.de/10.1145/301250.301317},
}

@article {PeteIso,
    AUTHOR = {Pete, G\'abor},
     TITLE = {A note on percolation on {$\Bbb Z^d$}: isoperimetric profile
              via exponential cluster repulsion},
   JOURNAL = {Electron. Commun. Probab.},
  FJOURNAL = {Electronic Communications in Probability},
    VOLUME = {13},
      YEAR = {2008},
     PAGES = {377--392},
      ISSN = {1083-589X},
   MRCLASS = {60K35 (82B43 82D30)},
  MRNUMBER = {2415145},
       DOI = {10.1214/ECP.v13-1390},
       URL = {https://doi-org.tum-eaccess.de/10.1214/ECP.v13-1390},
}

@article {PeresSousi,
    AUTHOR = {Peres, Yuval and Sousi, Perla},
     TITLE = {Mixing times are hitting times of large sets},
   JOURNAL = {J. Theoret. Probab.},
  FJOURNAL = {Journal of Theoretical Probability},
    VOLUME = {28},
      YEAR = {2015},
    NUMBER = {2},
     PAGES = {488--519},
      ISSN = {0894-9840,1572-9230},
   MRCLASS = {60J10},
  MRNUMBER = {3370663},
MRREVIEWER = {Ramon\ van Handel},
       DOI = {10.1007/s10959-013-0497-9},
       URL = {https://doi-org.tum-eaccess.de/10.1007/s10959-013-0497-9},
}

@article {Oliveira,
    AUTHOR = {Oliveira, Roberto Imbuzeiro},
     TITLE = {Mixing and hitting times for finite {M}arkov chains},
   JOURNAL = {Electron. J. Probab.},
  FJOURNAL = {Electronic Journal of Probability},
    VOLUME = {17},
      YEAR = {2012},
     PAGES = {no. 70, 12},
      ISSN = {1083-6489},
   MRCLASS = {60J27 (60J10)},
  MRNUMBER = {2968677},
       DOI = {10.1214/EJP.v17-2274},
       URL = {https://doi-org.tum-eaccess.de/10.1214/EJP.v17-2274},
}

@article {BergerBenjaminiYadin,
    AUTHOR = {Benjamini, Itai and Berger, Noam and Yadin, Ariel},
     TITLE = {Long-range percolation mixing time},
   JOURNAL = {Combin. Probab. Comput.},
  FJOURNAL = {Combinatorics, Probability and Computing},
    VOLUME = {17},
      YEAR = {2008},
    NUMBER = {4},
     PAGES = {487--494},
      ISSN = {0963-5483,1469-2163},
   MRCLASS = {60K35},
  MRNUMBER = {2433935},
MRREVIEWER = {Maria\ Deijfen},
       DOI = {10.1017/S0963548308008948},
       URL = {https://doi-org.tum-eaccess.de/10.1017/S0963548308008948},
}

@article{PenrosePisztora,
AUTHOR={Penrose, Mathew D. and Pisztora, Agoston},
TITLE={Large deviations for discrete and continuous percolation},
JOURNAL={Adv. in Appl. Probab.},
VOLUME={28},
YEAR={1996},
PAGES={29--52}
}

@article{EasoHutchcroft,
AUTHOR={Easo, Philip and Hutchcroft, Tom},
TITLE={Supercritical percolation on finite transitive graphs I: Uniqueness of the giant component},
JOURNAL={Duke Mathematical Journal},
VOLUMEE={173},
YEAR={2024},
NUMBER={13},
PAGES={2563--2618}
}

@phdthesis{montenegro2002faster,
  title={Faster mixing by isoperimetric inequalities},
  author={Montenegro, Ravi R.},
  year={2002},
  school={Yale University}
}

@article {AntalPisztora,
    AUTHOR = {Antal, Peter and Pisztora, Agoston},
     TITLE = {On the chemical distance for supercritical {B}ernoulli
              percolation},
   JOURNAL = {Ann. Probab.},
  FJOURNAL = {The Annals of Probability},
    VOLUME = {24},
      YEAR = {1996},
    NUMBER = {2},
     PAGES = {1036--1048},
      ISSN = {0091-1798,2168-894X},
   MRCLASS = {60K35 (82B43)},
  MRNUMBER = {1404543},
MRREVIEWER = {Mathew\ D.\ Penrose},
       DOI = {10.1214/aop/1039639377},
       URL = {https://doi-org.tum-eaccess.de/10.1214/aop/1039639377},
}

@article {KestenGWT,
    AUTHOR = {Kesten, Harry},
     TITLE = {Subdiffusive behavior of random walk on a random cluster},
   JOURNAL = {Ann. Inst. H. Poincar\'e{} Probab. Statist.},
  FJOURNAL = {Annales de l'Institut Henri Poincar\'e. Probabilit\'es et
              Statistique},
    VOLUME = {22},
      YEAR = {1986},
    NUMBER = {4},
     PAGES = {425--487}
}

@article{BergerTokushige,
      title={Scaling limits for random walks on long range percolation clusters}, 
      author={Noam Berger and Yuki Tokushige},
      year={2024},
      eprint={2403.18532},
      archivePrefix={arXiv},
      primaryClass={math.PR},
      journal={Preprint available at https://arxiv.org/abs/2403.18532}, 
}

@book {HighDim,
    AUTHOR = {Vershynin, Roman},
     TITLE = {High-dimensional probability},
    SERIES = {Cambridge Series in Statistical and Probabilistic Mathematics},
    VOLUME = {47},
      NOTE = {An introduction with applications in data science,
              With a foreword by Sara van de Geer},
 PUBLISHER = {Cambridge University Press, Cambridge},
      YEAR = {2018},
     PAGES = {xiv+284},
      ISBN = {978-1-108-41519-4},
   MRCLASS = {60-01 (60B05 60B20 60E15 60Fxx 62H25)},
  MRNUMBER = {3837109},
MRREVIEWER = {Sasha\ Sodin},
       DOI = {10.1017/9781108231596},
       URL = {https://doi-org.tum-eaccess.de/10.1017/9781108231596},
}

@article{SmallWorldLRP,
      title={Random walk on the small-world network model in 3 or more dimensions}, 
      author={Zsuzsanna Baran and Jonathan Hermon and Andela \v{S}arkovi\'{c} and Allan Sly and Perla Sousi},
      journal={Preprint available at https://arxiv.org/abs/2609.12056},
      year={2026},
}

@article {CroydonCanKumagai,
    AUTHOR = {Can, V. H. and Croydon, D. A. and Kumagai, T.},
     TITLE = {Spectral dimension of simple random walk on a long-range
              percolation cluster},
   JOURNAL = {Electron. J. Probab.},
  FJOURNAL = {Electronic Journal of Probability},
    VOLUME = {27},
      YEAR = {2022},
     PAGES = {Paper No. 56, 37},
      ISSN = {1083-6489},
   MRCLASS = {60K35 (35K05 60J35 60J74 60K37)},
  MRNUMBER = {4417202},
MRREVIEWER = {Elena\ Kosygina},
       DOI = {10.1214/22-ejp783},
       URL = {https://doi-org.tum-eaccess.de/10.1214/22-ejp783},
}
\end{document}